\documentclass[12pt]{amsart}
\usepackage{amssymb,graphicx,mathtools,enumerate}
\usepackage{amsmath,tabmac_avi}
\usepackage[pdfencoding=auto,psdextra]{hyperref}
\usepackage{tikz}
\usetikzlibrary{calc}
\usetikzlibrary{decorations.pathreplacing}

\oddsidemargin \evensidemargin
   \title{A Proof of the Extended Delta Conjecture}

   \author{J. Blasiak}
   \author{M. Haiman}
   \author{J. Morse}
   \author{A. Pun}
   \author{G. H. Seelinger}

   \address[Blasiak]{Dept.\ of Mathematics\\
            Drexel University \\
            Philadelphia, PA}
   \email{jblasiak@gmail.com}

   \address[Haiman]{Dept.\ of Mathematics\\
            University of California\\
            Berkeley, CA}
   \email{mhaiman@math.berkeley.edu}

   \address[Morse]{Dept.\ of Mathematics\\
   University of Virginia\\
   Charlottesville, VA}
   \email{morsej@virginia.edu}

   \address[Pun]{Dept.\ of Mathematics\\
   University of Virginia\\
   Charlottesville, VA}
   \email{ayp6e@virginia.edu}

   \address[Seelinger]{Dept.\ of Mathematics\\
   University of Virginia\\
   Charlottesville, VA}
   \email{ghs9ae@virginia.edu}

   \thanks{Authors were supported by NSF Grants DMS-1855784 (J.~B.)
    and DMS-1855804 (J.~M. and G.~S.).}

\newtheorem{thm}{Theorem}[subsection]
\newtheorem{lemma}[thm]{Lemma}
\newtheorem{prop}[thm]{Proposition}
\newtheorem{cor}[thm]{Corollary}

\theoremstyle{definition}
\newtheorem{defn}[thm]{Definition}

\theoremstyle{remark}
\newtheorem{example}[thm]{Example}

\newtheorem{remark}[thm]{Remark}

\newcommand{\kk}{{\mathbf k}}
\newcommand{\aA}{{\mathbf a}}
\newcommand{\bb}{{\mathbf b}}

\newcommand{\rr}{{\mathbf r}}
\newcommand{\sS}{{\mathbf s}}

\newcommand{\Dyck}{\mathbf {D}}
\newcommand{\LD}{{\mathbf {L}}}

\newcommand{\NN}{{\mathbb N}}
\newcommand{\QQ}{{\mathbb Q}}

\newcommand{\RR}{{\mathbb R}}
\newcommand{\ZZ}{{\mathbb Z}}

\newcommand{\Hbold}{{\mathbf H}}

\newcommand{\sigmabold}{{\boldsymbol \sigma }}

\newcommand{\Ecal}{{\mathcal E}}

\newcommand{\Lcal}{{\mathcal L}}
\newcommand{\Acal}{{\mathcal A}}

\newcommand{\Htild}{\tilde{H}}

\DeclareMathOperator{\dinv}{dinv}

\DeclareMathOperator{\pol}{pol}

\DeclareMathOperator{\GL}{GL}
\DeclareMathOperator{\SL}{SL}

\DeclareMathOperator{\SSYT}{SSYT}
\DeclareMathOperator{\NSYT}{RST}
\DeclareMathOperator{\area}{area}
\DeclareMathOperator{\Ad}{Ad}
\DeclareMathOperator{\wt}{wt}
\DeclareMathOperator{\op}{op}

\newcommand{\sump}[2]{{{\mathmakebox[0.3in]{\sum\limits_{#1}^{#2}}}
\hskip-0.13in\phantom{x}^{\#}}}

\begin{document}

\subjclass[2010]{Primary: 05E05; Secondary: 16T30}

\begin{abstract}
We prove the Extended Delta Conjecture of Haglund, Remmel, and Wilson,
a combinatorial formula for $\Delta _{h_l}\Delta' _{e_k} e_{n}$, where
$\Delta' _{e_k}$ and $\Delta_{h_l}$ are Macdonald eigenoperators and
$e_n$ is an elementary symmetric function.  We actually prove a
stronger identity of infinite series of $\GL_m$ characters
expressed in terms of LLT series. This is achieved through new results
in the theory of the Schiffmann algebra and its action on the algebra
of symmetric functions.
\end{abstract}

\maketitle

\section{Introduction}
\label{s:intro}

We prove the {\em Extended Delta Conjecture} of Haglund, Remmel and
Wilson~\cite{HagRemWil18} by adapting methods from our work in
\cite{paths} on a generalized Shuffle Theorem and proving new results
about the action of the elliptic Hall algebra on symmetric functions.
As in \cite{paths}, we reformulate the conjecture as the polynomial
truncation of an identity of infinite series of $\GL_m$ characters,
expressed in terms of LLT series.  We then prove the stronger infinite
series identity using a Cauchy identity for non-symmetric
Hall-Littlewood polynomials.

The conjecture stemmed from studies of the diagonal coinvariant
algebra ${\rm DR}_n$ in two sets of $n$ variables, whose character as
a doubly graded $S_{n}$ module has remarkable links with both
classical combinatorial enumeration and the theory of Macdonald
polynomials.  It was shown in \cite{Haiman02} that this character is
neatly given by the formula $\Delta'_{e_{n-1}} e_n$, where
$\Delta'_{f}$ is a certain eigenoperator on Macdonald polynomials and
$e_{n}$ is the $n$-th elementary symmetric function.

The {\em Shuffle Theorem}, conjectured in~\cite{HHLRU} and proven
by Carlsson and Mellit in~\cite{CarlMell18}, gives a combinatorial
expression for $\Delta'_{e_{n-1}}e_n$ in terms of Dyck paths---that
is, lattice paths from $(0,n)$ to $(n,0)$ that lie weakly below the
line segment connecting these two points.

An expanded investigation led Haglund, Remmel and
Wilson~\cite{HagRemWil18} to the {\it Delta Conjecture}, a
combinatorial prediction for $\Delta'_{e_{k}}e_n$, for all $0\leq
k<n$.  This led to a flurry of activity (e.g.~\cite{DaddIraVan19,
GaHaReY, HagRemWil18, HagRhSh18, 
HagRhSh19,QW,
Rhoades18, Romero17, Wil16,
Zabrocki19b}), including a conjecture by Zabrocki
~\cite{Zabrocki19} that $\Delta'_{e_{k}}e_n$ captures the character of
the super-diagonal coinvariant ring ${\rm SDR}_n$, a deformation of
${\rm DR}_n$ involving the addition of a set of anti-commuting
variables.

The Delta Conjecture has been extended in two directions.  One gives a
{\it Compositional} generalization, just proved by D'Adderio and
Mellit~\cite{DaddMe}.  The other involves a second eigenoperator
$\Delta_{h_l}$, where $h_l$ is the $l$-th homogeneous symmetric
function.  The {\it Extended Delta Conjecture}~\cite[Conjecture
7.4]{HagRemWil18} is, for $l\geq 0$ and $1\leq k\leq n$,
\begin{equation}\label{e:dc}
\Delta_{h_l}\Delta'_{e_{k-1}} e_n = \langle z^{n-k}\rangle
\sum_{\lambda \in \Dyck_{n+l}} \sum_{P\in\LD_{n+l,l}(\lambda)}
q^{\dinv(P)}t^{\area(\lambda)} x^{\wt_+(P)}
\prod_{r_{i}(\lambda)=r_{i-1}(\lambda)+1} \left(1+ z\,
t^{-r_i(\lambda)}\right) \,,
\end{equation}
in which $\lambda $ is a Dyck path and $P$ is a certain type of
labelling of $\lambda $ (see \S~\ref{s:extended-delta} for full
definitions).  D'Adderio, Iraci and Wyngaerd proved the Schr\"oder case and the $t = 0$
specialization of the conjecture~\cite{DaddIraVan19a,DaddIraVan19}; Qiu and
Wilson~\cite{QW} reformulated the conjecture and established the $q=0$
specialization as well.

Let us briefly outline the steps by which we prove \eqref{e:dc}.

Feigin--Tsymbauliak~\cite{FeigTsym11} and
Schiffmann--Vasserot~\cite{SchiVass13} constructed an action of the
elliptic Hall algebra $\Ecal $ of Burban and
Schiffmann~\cite{BurbSchi12} on the algebra of symmetric functions.
The operators $\Delta _{f}$ and $\Delta '_{f}$ are part of the $\Ecal$
action.  In Theorem~\ref{t:ourLHS}, we use this to reformulate the
left hand side of~\eqref{e:dc} as the polynomial part of an explicit
infinite series of virtual $\GL_{m}$ characters with coefficients in
$\QQ (q,t)$.
The proof of Theorem~\ref{t:ourLHS} relies on a symmetry
between distinguished elements of $\Ecal $ introduced by
Negut~\cite{Negut14} and their transposes given in
Proposition~\ref{prop:D=E}.  

In Theorem~\ref{t:rhs_reformulation}, we also reformulate the right
hand side of~\eqref{e:dc} as the polynomial part of an infinite
series, in this case expressed in terms of the LLT series introduced
by Grojnowski and Haiman in \cite{GrojHaim07}.

With~\eqref{e:dc} now rewritten as an equality between the polynomial
parts of two expressions in~\eqref{e:infseriesidentity}, we ultimately
arrive at Theorem~\ref{thm:1}---an identity of infinite series of
$\GL_m$ characters which implies the Extended Delta Conjecture by
taking the polynomial part on each side.

Although the Extended Delta Conjecture and the Compositional Delta
Conjecture both imply the Delta Conjecture, they generalize it in
different directions and our methods are quite different from those of
D'Adderio and Mellit.  It would be interesting to know whether a
common generalization is possible.

\section{The Extended Delta Conjecture}
\label{s:extended-delta}

The Extended Delta Conjecture equates a ``symmetric function side'',
involving the action of a Macdonald operator on an elementary
symmetric function, with a ``combinatorial side''.  We begin by
recalling the definitions of these two quantities.

\subsection{Symmetric function side}
\label{ss:symmetric-side}

Integer partitions are written \(\lambda = (\lambda _{1}\geq \cdots
\geq \lambda _{l})\), sometimes with trailing zeroes allowed.  We set
$|\lambda | = \lambda _{1}+\cdots +\lambda _{l}$ and let $\ell(\lambda
)$ be the number of non-zero parts.  We identify a partition $\lambda$
with its French style Ferrers shape, the set of lattice squares (or
{\em boxes}) with northeast corner in the set
\begin{equation}\label{e:Young-diagram}
\{(i,j)\mid 1\leq j\leq \ell(\lambda ),\; 1\leq i \leq \lambda _{j} \}.
\end{equation}
The {\em shape generator} of $\lambda $ is the polynomial
\begin{equation}\label{e:B-lambda}
B_{\lambda }(q,t) = \sum _{(i,j)\in \lambda} q^{i-1}\, t^{j-1}.
\end{equation}

Let $\Lambda = \Lambda _{\kk }(X)$ be the algebra of symmetric
functions in an infinite alphabet of variables $X =
x_{1},x_{2},\ldots$, with coefficients in the field $\kk = \QQ (q,t)$.
We follow the notation of Macdonald~\cite{Macdonald95} for the graded
bases of $\Lambda $.  Basis elements are indexed by a partition
$\lambda $ and have homogeneous degree $|\lambda |$.  Examples include
the elementary symmetric functions $e_{\lambda } = e_{\lambda
_{1}}\cdots e_{\lambda _{k}}$, complete homogeneous symmetric
functions $h_{\lambda } = h_{\lambda _{1}}\cdots h_{\lambda _{k}}$,
power-sums $p_{\lambda } = p_{\lambda _{1}}\cdots p_{\lambda _{k}}$,
monomial symmetric functions $m_{\lambda }$, and Schur functions
$s_{\lambda }$.

As is conventional, $\omega \colon \Lambda \rightarrow \Lambda $
denotes the $\kk $-algebra involution defined by $ \omega s_{\lambda}
= s_{\lambda^*}$,  where $\lambda ^{*}$ denotes the transpose of $\lambda$,
and $\langle -, - \rangle$ denotes the symmetric bilinear
inner product such that \(\langle s_{\lambda },s_{\mu } \rangle =
\delta _{\lambda ,\mu }\).

The basis of modified Macdonald polynomials, $\Htild_{\mu }(X;q,t)$,
is defined~\cite{GarsHaim96} from the integral form Macdonald
polynomials $J_{\mu }(X;q,t)$ of \cite{Macdonald95} using the device
of {\em plethystic evaluation}.  For an expression $A$ in terms of
indeterminates, such as a polynomial, rational function, or formal
series, $p_{k}[A]$ is defined to be the result of substituting $a^{k}$
for every indeterminate $a$ occurring in $A$.  We define $f[A]$ for
any $f\in \Lambda $ by substituting $p_{k}[A]$ for $p_{k}$ in the
expression for $f$ as a polynomial in the power-sums $p_{k}$, so that
$f\mapsto f[A]$ is a homomorphism.  The variables $q, t$ from our
ground field $\kk $ count as indeterminates.  The modified Macdonald
polynomials are defined by
\begin{equation}\label{e:H-tilde}
\Htild _{\mu }(X;q,t) = t^{n(\mu )} J_{\mu
}\left[\frac{X}{1-t^{-1}};q,t^{-1}\right],
\end{equation}
where
\begin{equation}\label{e:n-lambda}
n(\mu  ) = \sum _{i} (i-1)\mu _{i}
\end{equation}
For any symmetric function $f\in \Lambda $, let $f[B]$ denote the
eigenoperator on the basis $\{\tilde{H_{\mu }} \}$ of $\Lambda $ such
that
\begin{equation}\label{e:eigen-operators}
f[B]\, \Htild _{\mu } = f[B_{\mu }(q,t)]\, \Htild _{\mu }\,.
\end{equation}

The left hand side of~\eqref{e:dc} is expressed in the notation
of~\cite{HagRemWil18}, where $\Delta_f=f[B]$ and $\Delta_f'=f[B-1]$.
Hence, the {\bf symmetric function side of the Extended Delta
Conjecture} is
\begin{equation}\label{e:sside}
h_l[B]e_{k-1}[B-1] e_n\,.
\end{equation}

\subsection{The combinatorial side}
\label{ss:combinatorial-side}

The right hand side of the Extended Delta Conjecture~\eqref{e:dc} 
is a combinatorial generating function that counts labelled
lattice paths.

\begin{defn}\label{def:dyck-paths}
A \emph{Dyck path} is a south-east lattice path lying weakly below the
line segment connecting the points \((0,N)\) and \((N,0)\).  The set
of such paths is denoted \(\Dyck_N\).  The \emph{staircase path}
\(\delta\) is the Dyck path alternating between south and east steps.

Each \(\lambda \in \Dyck_N\) has \(\area(\lambda)=|\delta
/\lambda|\) defined to be the number of lattice squares lying above
\(\lambda\) and below \(\delta\).  Let \(r_i(\lambda)\) be the area
contribution from squares in the \(i\)-th row, numbered from north to
south; in other words, $r_{i}$ is the distance from the \(i\)-th south
step of \(\lambda\) to the \(i\)-th south step of \(\delta\).  Note
that
\begin{equation}
r_1(\lambda)=0,\qquad r_i(\lambda)\leq r_{i-1}(\lambda)+1 \quad
\text{for }i>1, \quad\text{and}\quad \sum_{i=1}^{N}
r_i(\lambda)=|\delta/\lambda|\,.
\end{equation}
\end{defn}

\begin{defn}\label{def:labeled-dyck-paths}
A \emph{labelling} \(P = (P_1,\dots,P_N) \in \NN^N\) attaches a label
in \(\NN=\{0,1,\ldots\}\) to each south step of \(\lambda \in
\Dyck_N\) so that the labels increase from north to south along
vertical runs of south steps, as shown in Figure~\ref{fig:PF}. The set
of labellings is denoted by \(\LD_N(\lambda)\), or simply
\(\LD(\lambda)\).  Given $0\leq l<N$, a \emph{partial labelling} of
\(\lambda \in \Dyck_{N}\) is a labelling where \(0\) occurs exactly
$l$ times and never on the line \(x=0\). We denote the set of these
partial labellings by \(\LD_{N,l}(\lambda)\).

To each labelling \(P \in \LD(\lambda)\) is associated a statistic
\(\dinv(P)\), defined to be the number of pairs \((i < j)\) such that
either
\begin{equation} \label{e:rconditions}
\begin{cases}
   r_i(\lambda)=r_j(\lambda) \text{ and } P_i<P_j \;\;{\rm or}\\
   r_i(\lambda)=r_j(\lambda)+1 \text{ and } P_i>P_j \,.
\end{cases}
\end{equation}
The {\it weight} of a labelling \(P\) is defined so zero labels do not
contribute, by
\begin{equation}\label{e:xweight}
x^{\wt_+(P)} = \prod_{i\in [N]: P_i\neq 0}x_{P_i}\,.
\end{equation}
This is equivalent to letting \(x_0 = 1\) in \(x^{\wt(P)} := \prod_{i
\in [N]} x_{P_i} \).
\end{defn}

The above defines the right hand side of~\eqref{e:dc}, with $\langle
z^{n-k}\rangle$ denoting the coefficient of $z^{n-k}$.

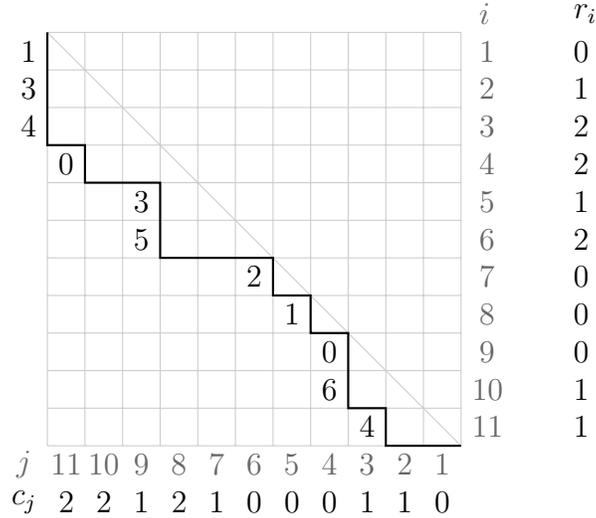
\begin{figure}
\begin{tikzpicture}[xscale = 0.5,yscale = 0.5]
	\draw[black!20] (0,0) grid (11,11);
	\draw[black!20] (0,11)--(11,0);
	\draw[thick] (0,11)--(0,8)--(1,8)--(1,7)--(3,7)--(3,5)--(6,5)--(6,4)--(7,4)--(7,3)--(8,3)--(8,1)--(9,1)--(9,0)--(11,0);
	\node[left] at (0,10.5) {$1$};
	\node[left] at (0,9.5) {$3$};
	\node[left] at (0,8.5) {$4$};
	\node[left] at (1,7.5) {$0$};
	\node[left] at (3,6.5) {$3$};
	\node[left] at (3,5.5) {$5$};
	\node[left] at (6,4.5) {$2$};
	\node[left] at (7,3.5) {$1$};
	\node[left] at (8,2.5) {$0$};
	\node[left] at (8,1.5) {$6$};
	\node[left] at (9,0.5) {$4$};
	\node[black!60] at (10.5,-0.5) {$1$};
	\node[black!60] at (9.5,-0.5) {$2$};
	\node[black!60] at (8.5,-0.5) {$3$};
	\node[black!60] at (7.5,-0.5) {$4$};
	\node[black!60] at (6.5,-0.5) {$5$};
	\node[black!60] at (5.5,-0.5) {$6$};
	\node[black!60] at (4.5,-0.5) {$7$};
	\node[black!60] at (3.5,-0.5) {$8$};
	\node[black!60] at (2.5,-0.5) {$9$};
	\node[black!60] at (1.5,-0.5) {$10$};
	\node[black!60] at (.5,-0.5) {$11$};
	\node[left,black!60] at (0,-0.5) {$j\, $};
	\node[left] at (11,-1.5) {$0$};
	\node[left] at (10,-1.5) {$1$};
	\node[left] at (9,-1.5) {$1$};
	\node[left] at (8,-1.5) {$0$};
	\node[left] at (7,-1.5) {$0$};
	\node[left] at (6,-1.5) {$0$};
	\node[left] at (5,-1.5) {$1$};
	\node[left] at (4,-1.5) {$2$};
	\node[left] at (3,-1.5) {$1$};
	\node[left] at (2,-1.5) {$2$};
	\node[left] at (1,-1.5) {$2$};
	\node[left] at (0,-1.5) {$c_j$};
	\node[right,black!60] at (11.2,11.5){$i$};
	\node[right,black!60] at (11.2,10.5){$1$};
	\node[right,black!60] at (11.2,9.5) {$2$};
	\node[right,black!60] at (11.2,8.5) {$3$};
	\node[right,black!60] at (11.2,7.5) {$4$};
	\node[right,black!60] at (11.2,6.5) {$5$};
	\node[right,black!60] at (11.2,5.5) {$6$};
	\node[right,black!60] at (11.2,4.5) {$7$};
	\node[right,black!60] at (11.2,3.5) {$8$};
	\node[right,black!60] at (11.2,2.5) {$9$};
	\node[right,black!60] at (11,1.5) {$10$};
	\node[right,black!60] at (11,0.5) {$11$};
	\node at (14.3,11.5) {$r_i$};
	\node at (14.2,10.5) {$0$};
	\node at (14.2,9.5) {$1$};
	\node at (14.2,8.5) {$2$};
	\node at (14.2,7.5) {$2$};
	\node at (14.2,6.5) {$1$};
	\node at (14.2,5.5) {$2$};
	\node at (14.2,4.5) {$0$};
	\node at (14.2,3.5) {$0$};
	\node at (14.2,2.5) {$0$};
	\node at (14.2,1.5) {$1$};
	\node at (14.2,0.5) {$1$};
\end{tikzpicture}
\caption{\label{fig:PF}
A path $\lambda $ and partial labelling
$P\in\LD_{11,2}(\lambda)$, with \(\area(\lambda)=10\), \(\dinv(P) =
15\), \(x^{\wt_+(P)} = x_1^2 x_2 x_3^2 x_4^2 x_5 x_6\), and
\(x^{\wt(P)} = x_0^2 x_1^2 x_2 x_3^2 x_4^2 x_5 x_6\).}
\end{figure}

\begin{remark}\label{r:reconcilenotation}
In~\cite{HagRemWil18}, a Dyck path is a north-east lattice path lying
weakly above the line segment connecting \((0,0)\) and \((N,N)\), and
labellings increase from south to north along vertical runs.  After
reflecting the picture about a horizontal line, our conventions on
paths, labellings, and the definition of $\dinv (P)$ match those
in~\cite{HagRemWil18}.  Separately,~\cite{HHLRU} uses the same
conventions that we do for Dyck paths, but defines labellings to
increase from south to north, and defines $\dinv (P)$ with the
inequalities in~\eqref{e:rconditions} reversed.  However, since the
sum
\begin{equation} \label{eq:dinv-wt-gen-fn} \sum_{P \in \LD(\lambda)}
q^{\dinv(P)} x^{\wt(P)}\,
\end{equation}
is a symmetric function~\cite{HHLRU}, it is unchanged if we reverse
the ordering on labels, after which the conventions in~\cite{HHLRU}
agree with those used here.
\end{remark}

We prefer another slight modification based on the following lemma
which was mentioned in~\cite{HagRemWil18} without details.

\begin{lemma}\label{lem:r-prod=c-prod}
For any Dyck path \(\lambda\in\Dyck_N\), we have
\begin{equation} \label{eq:ri-is-ci}
\prod_{\substack{1<i\leq N\\r_{i}(\lambda)=r_{i-1}(\lambda)+1}} \left(1+
z\, t^{-r_i(\lambda)}\right) = \prod_{\substack{1 <i \leq N \\
c_i(\lambda) = c_{i-1}(\lambda)+1}} (1+z\, t^{-c_i(\lambda)}) \,,
\end{equation}
where \(c_i(\lambda)=r_i(\lambda^*)\) is the contribution to
\(|\delta/\lambda|\) from boxes in the $i$-th column, numbered from
right to left.
\end{lemma}

\begin{proof}
The condition $r_{i}(\lambda)=r_{i-1}(\lambda)+1$ means that
$\lambda $ has consecutive south steps in rows $i-1$ and $i$ with no
intervening east step.  Similarly, $c_{i}(\lambda )=c_{i-1}(\lambda )+1$
if and only if $\lambda $ has consecutive east steps in columns $i-1$
and $i$ (numbered right to left).  Consider the word formed by listing
the steps in $\lambda $ in the southeast direction from $(0,N)$ to
$(N,0)$, as shown here for the example in Figure~\ref{fig:PF}.
\begin{equation*}
    \begin{tikzpicture}
      \newcounter{parennum}; \setcounter{parennum}{0}; \foreach \lab
      in
      {"S","S","S","E","S","E","E","S","S","E","E","E","S","E","S","E","S","S","E","S","E","E"}
      { \node at (0.3*\theparennum,0)
        {\pgfmathparse{\lab}\pgfmathresult};
        \addtocounter{parennum}{1}; }
      \draw[yshift=-3mm] (0,0) --
      (0,-0.3) -- (3.3,-0.3) -- (3.3,0);
      \draw[yshift=-3mm] (0.3,0) --
      (0.3,-0.2) -- (1.8,-0.2) -- (1.8,0);
      \draw[yshift=-3mm] (2.1,0)
      -- (2.1,-0.2) -- (3,-0.2) -- (3,0);
      \draw[yshift=-3mm]
      (4.8,0) -- (4.8,-0.2) -- (6.3,-0.2) -- (6.3,0);
    \end{tikzpicture}
\end{equation*}
Treating south and east steps as left and right parentheses,
each south step pairs with an east step to its right, and we have
$r_{i}(\lambda ) = c_{j}(\lambda )$ if the $i$-th south step (numbered
left to right) pairs with the $j$-th east step (numbered right to
left).  Furthermore, the leftmost member of each double south step
pairs with the rightmost member of a double east step, as indicated in
the word displayed above.

Since each index \(i-1\) such that
\(r_i(\lambda)= r_{i-1}(\lambda)+1\) pairs with an index $j-1$
such that
\(c_j(\lambda) = c_{j-1}(\lambda)+1\), we have
\begin{equation}\label{e:rprod-cprod}
\prod_{\substack{1 < i \leq N \\ r_i(\lambda) = r_{i-1}(\lambda)+1}}
(1+z\, t^{-r_{i-1}(\lambda)-1}) 
\,\,\, = \!\!\! \prod_{\substack{1 < j \leq N \\
c_j(\lambda) = c_{j-1}(\lambda)+1}} (1+z\, t^{-c_{j-1}(\lambda)-1}) \,.
\end{equation}
Now \eqref{eq:ri-is-ci} follows.
\end{proof}

Setting $N=n+l$ and $m=k+l$, the right hand side of \eqref{e:dc}, 
or the {\bf combinatorial side of the Extended Delta Conjecture}, 
is equal to
\begin{equation}\label{eq:EE}
 \langle z^{N-m} \rangle \sum_{\substack{\lambda\in\Dyck_{N}\\
P\in\LD_{N,l}(\lambda)}}
t^{|\delta/\lambda|} \, q^{\dinv(P)} \,
x^{\wt_+(P)} \prod_{\substack{1 <i \leq N \\ c_i(\lambda)=
c_{i-1}(\lambda)+1}} (1+z\, t^{-c_i(\lambda)}) \,.
\end{equation}

\section{Background on the Schiffmann algebra \texorpdfstring{$\Ecal$}{E}}
\label{s:reform-symmetric}
\subsection{}
From work of Feigin and Tsymbauliak~\cite{FeigTsym11} and Schiffmann
and Vasserot~\cite{SchiVass13}, we know that the operators $f[B]$ in
\eqref{e:sside} form part of an action of the elliptic Hall algebra
$\Ecal $ of Burban and Schiffmann~\cite{BurbSchi12,Schiffmann12}, or
{\em Schiffmann algebra} for short, on the algebra of symmetric
functions.  In~\cite{paths}, we used this action to express the
symmetric function side of a generalized Shuffle Theorem as the
polynomial part of an explicit infinite series of $\GL_{l}$ characters.
Here we derive a similar expression (Theorem~\ref{t:ourLHS}) for the
symmetric function side \eqref{e:sside} of the Extended Delta
Conjecture.

For this purpose, we need a deeper study of the Schiffmann algebra
than we did in~\cite{paths}, where a fragment of the theory was
enough.  We start with a largely self-contained description of $\Ecal
$ and its action on $\Lambda $, although we occasionally refer
to~\cite{paths} for the restatements of results
from~\cite{BurbSchi12,Schiffmann12,SchiVass13} in our notation, and
for some proofs.  A precise translation between our notation and that
of~\cite{BurbSchi12,Schiffmann12,SchiVass13} can be found
in~\cite[eq.~(25)]{paths}.  In the presentation of $\Ecal $ and its
action on $\Lambda $, we freely use plethystic substitution, defined
in \S \ref{ss:symmetric-side}.  Indeed, the ability to do so is a
principal reason why we prefer the notation used here to that in the
foundational papers on the Schiffmann algebra.

\subsection{Description of \texorpdfstring{$\Ecal$}{E}}
\label{ss:Schiffmann}

Let $\kk = \QQ (q,t)$, as in \S \ref{s:extended-delta}.  The
Schiffmann algebra $\Ecal $ is generated by a central Laurent
polynomial subalgebra $F = \kk [c_{1}^{\pm 1}, c_{2}^{\pm 1}]$ and a
family of subalgebras $\Lambda _{F}(X^{m,n})$ isomorphic to the
algebra of symmetric functions $\Lambda _{F}(X)$ over $F$, one for
each pair of coprime integers $(m,n)$.  These are subject to defining
relations spelled out below.

For any algebra $A$ containing a copy of $\Lambda $, there is an
{\em adjoint action} of $\Lambda $ on $A$ arising from the Hopf
algebra structure of $\Lambda $.  Using two formal alphabets $X$ and
$Y$ to distinguish between the tensor factors in $\Lambda \otimes
\Lambda \cong \Lambda (X) \Lambda (Y)$, the coproduct and antipode for
the Hopf algebra structure are given by the plethystic substitutions
\begin{equation}\label{e:Lambda-coprod}
\Delta f = f[X+Y],\qquad S(f) = f[-X].
\end{equation}
The adjoint action of $f\in \Lambda $ on $\zeta \in A$ is then given by
\begin{equation}\label{e:Ad-action}
(\Ad f)\, \zeta = \sum_{i} f_{i} \, \zeta \, g_{i},\quad \text{where}
\quad f[X-Y] = \sum_{i} f_{i}(X)g_{i}(Y),
\end{equation}
since the formula on the right is another way to write $(1\otimes S)\Delta
f = \sum _{i} f_{i}\otimes g_{i}$.  More explicitly, we have
\begin{equation}\label{e:Ad-examples}
(\Ad p_{n})\, \zeta =[p_{n},\zeta ]\quad \text{and} \quad (\Ad
h_{n})\, \zeta = \sum _{j+k=n} (-1)^{k} h_{j}\, \zeta \, e_{k}.
\end{equation}
The last formula can be expressed for all $n$ at once as a generating
function identity
\begin{equation}\label{e:Ad-Omega}
(\Ad \Omega [zX])\, \zeta = \Omega [zX]\, \zeta\,  \Omega [-zX],
\end{equation}
where 
\begin{equation}\label{e:Omega}
\Omega (X) = \sum _{n = 0}^{\infty } h_{n}(X).
\end{equation}

We fix notation for the quantities
\begin{equation}\label{e:M-and-M-hat}
M=(1-q)(1-t),\qquad \widehat{M} = (1-(q\, t)^{-1})M,
\end{equation}
which play a role in the presentation of $\Ecal $ and will be referred
to again later.

\subsubsection{Basic structure and symmetries}
\label{sss:structure}

The algebra $\Ecal$ is $\ZZ^2$ graded with the central subalgebra $F$
in degree $(0,0)$ and $f(X^{m,n})$ in degree $(dm,dn)$ for $f(X) $ of
degree $d$ in $\Lambda (X)$.

The universal central extension $\widehat{\SL_{2}(\ZZ )}\rightarrow
\SL _{2}(\ZZ )$ acts on the set of tuples
\begin{equation}\label{e:Z2-hat}
\{(m,n,\theta )\in (\ZZ ^{2}\setminus {\boldsymbol 0})\times \RR \mid 
\text{$\theta $ is a value of $\arg (m+in)$} \},
\end{equation}
lifting the $\SL _{2}(\ZZ )$ action on pairs $(m,n)$, with the central
subgroup $\ZZ $ generated by the `rotation by $2\pi $' map
$(m,n,\theta )\mapsto (m,n,\theta +2\pi )$.  The group $\widehat{\SL
_{2}(\ZZ )}$ acts on $\Ecal $ by $\kk $-algebra automorphisms,
compatibly with the action of $\SL _{2}(\ZZ )$ on the grading group
$\ZZ ^{2}$.  Before giving the defining relations of $\Ecal $, we
specify how $\widehat{\SL _{2}(\ZZ )}$ acts on the generators.

For each pair of coprime integers $(m,n)$, we introduce a family of
alphabets $X^{m,n}_\theta$, one for each value $\theta $ of $\arg
(m+in)$, related by
\begin{equation}\label{e:Xmn-theta}
X^{m,n}_{\theta+2\pi} = c_1^m c_2^n X^{m,n}_{\theta}.
\end{equation}
We make the convention that $X^{m,n}$ without a subscript means
$X^{m,n}_{\theta }$ with $\theta \in (-\pi ,\pi ]$.
For comparison, the implied convention in \cite{paths} is 
$\theta\in[-\pi,\pi)$.
The subalgebra
$\Lambda _{F}(X^{m,n}) = \Lambda _{F}(X^{m,n}_{\theta })$ only depends
on $(m,n)$ and so does not depend on the choice of branch for the
angle $\theta $.  We will also refer below to the subalgebras
$\Lambda_{\kk } (X^{m,n})$, which do depend on our choice of branch,
unless we specialize the central parameters $c_{1},c_{2}$ to elements
of $\kk $.

The $\widehat{\SL _{2}(\ZZ )}$ action is now given by $\rho \cdot
f(X^{m,n}_{\theta }) = f(X^{m',n'}_{\theta '})$ for $f(X)\in \Lambda
_{\kk }(X)$ where $\rho \in \widehat{\SL _{2}(\ZZ )}$ acts on the
indexing data in \eqref{e:Z2-hat} by $\rho \cdot (m,n,\theta ) =
(m',n',\theta ')$.  Note that if $m,n$ are coprime, then so are
$m',n'$.  The action on $F$ factors through the action of $\SL
_{2}(\ZZ )$ on the group algebra $\kk \cdot \ZZ ^{2}\cong F$.

For instance, the `rotation by $2\pi$' element $\rho \in \widehat{\SL
_{2}(\ZZ )}$ fixes $F$, and has $\rho \cdot f(X_\theta^{m,n}) =
f(X_{\theta+2\pi}^{m,n}) = f[c_{1}^{m} c_{2}^{n}X^{m,n}_{\theta }]$.
Thus $\rho $ coincides with multiplication by $c_{1}^{r} c_{2}^{s}$ in
degree $(r,s)$, and automatically preserves all relations that respect
the $\ZZ ^{2}$ grading.

We now turn to the defining relations of $\Ecal$.  Apart from the
relations implicit in $F = \kk [c_{1}^{\pm 1},c_{2}^{\pm 1}]$ being
central and each $\Lambda _{F}(X^{m,n})$ being isomorphic to $\Lambda
_{F}(X)$, these fall into three families: Heisenberg relations,
internal action relations and axis-crossing relations.

\subsubsection{Heisenberg relations}
\label{sss:Heisenberg}

Each pair of subalgebras $\Lambda _{F}(X^{m,n})$ and $\Lambda
_{F}(X^{-m,-n})$ in degrees along opposite rays in $\ZZ ^{2}$ satisfy
Heisenberg relations
\begin{equation}\label{e:Heisenberg}
[p_k(X_\theta^{-m,-n}),\, p_l(X_{\theta+\pi}^{m,n})] = \delta_{k,l}\,
k \, p_k[(c_{1}^{m}c_{2}^{n}-1) / \widehat{M}],
\end{equation}
where $\widehat{M}$ is given by \eqref{e:M-and-M-hat}.  As an
exercise, the reader can check, using \eqref{e:Xmn-theta}, that the
relations in \eqref{e:Heisenberg} are consistent with swapping the
roles of $\Lambda _{F} (X^{m,n})$ and $\Lambda _{F}(X^{-m,-n})$.

\subsubsection{Internal action relations}
\label{sss:internal-action}

The internal action relations describe the adjoint action of each
$\Lambda _{F}(X^{m,n})$ on $\Ecal $.  For simplicity, we write these
relations, and also the axis-crossing relations below, with $\Lambda
_{F}(X^{1,0})$ distinguished.  The full set of relations is understood
to be given by closing the stated relations under the $\widehat{\SL
_{2}(\ZZ )}$ action.

Bearing in mind that $X^{m,n}$ means $X^{m,n}_{\theta }$ with $\theta
\in (-\pi ,\pi ]$, the relations for the internal action of $\Lambda
_{F}(X^{1,0})$ are:
\begin{equation}\label{e:internal-action}
\begin{aligned}
(\Ad f(X^{1,0}))\, p_1(X^{m,1}) & = (\omega f)[z] \Bigl| z^k \mapsto
p_1(X^{m+
k,1})\\
(\Ad f(X^{1,0}))\, p_1(X^{m,-1}) & = (\omega f)[-z] \Bigl| z^k \mapsto
p_1(X^{m+ k,-1})
\end{aligned}
\end{equation}

\subsubsection{Axis-crossing relations}
\label{sss:axis-crossing}

Again distinguishing $\Lambda _{F}(X^{1,0})$ and taking angles on the
branch $\theta \in (-\pi ,\pi ]$, the final set of relations is the
closure under the $\widehat{\SL _{2}(\ZZ )}$ action of
\begin{equation}\label{e: axis-crossing}
[p_1(X^{b,-1}),\, p_1(X^{a,1})] =
-\frac{e_{a+b}[-\widehat{M}X^{1,0}]}{\widehat{M}} \quad \text{for
$a+b>0$}.
\end{equation}
More generally, rotating this relation by $\pi$ determines
$[p_1(X^{b,-1}),\, p_1(X^{a,1})]$ for $a + b < 0$, and the Heisenberg
relations determine it when $a + b = 0$. Combining these gives
\begin{equation}\label{e: general axis-crossing}
[p_1(X^{b,-1}),\, p_1(X^{a,1})] = -\frac{1}{\widehat{M}}
\begin{cases}
      	e_{a+b}[-\widehat{M}X^{1,0}] & a+b>0\\
      	1-c_1^{-b}c_2 & a + b = 0\\
      	-c_1^{-b}c_2e_{-(a+b)}[-\widehat{M}X^{-1,0}] & a+b < 0\,.\\
\end{cases}
\end{equation}

\subsubsection{Further remarks}
\label{sss:remarks}

Define {\em upper} and {\em lower half} subalgebras $\Ecal^{* , >0},
\Ecal ^{*, <0} \subseteq \Ecal$ to be generated by the $\Lambda
_{F}(X^{m,n})$ with $n > 0$ or $n < 0$, respectively.  Using the
$\widehat{\SL _2(\ZZ )}$ image of the relations in \eqref{e: axis-crossing},
one can express any $e_{k}[-\widehat{M}
X^{m,n}]$ for $n>0$ in terms of iterated commutators of the elements
$p_{1}(X^{a,1})$.
This shows that $\{p_1(X^{a,1})\mid a \in \ZZ\}$
generates $\Ecal ^{*, >0}$ as an $F$-algebra.  Similarly,
$\{p_1(X^{a,-1})\mid a \in \ZZ\}$ generates $\Ecal ^{*, <0}$.

The internal action relations give the adjoint action of $\Lambda
_{F}(X^{1,0})$ on the space spanned by $\{p_1(X^{a,\pm 1})\mid a \in
\ZZ\}$.  Using the formula $(\Ad f)(\zeta _{1}\zeta _{2}) = \sum ((\Ad
f_{(1)})\zeta _{1})((\Ad f_{(2)})\zeta _{2})$, where $\Delta f = \sum
f_{(1)}\otimes f_{(2)}$ in Sweedler notation, this determines the
adjoint action of $\Lambda _{F}(X^{1,0})$ on $\Ecal ^{*, >0}$ and
$\Ecal ^{*, <0}$.  The Heisenberg relations give the adjoint action of
$\Lambda _{F}(X^{1,0})$ on $\Lambda _{F}(X^{-1,0})$, while $\Lambda
_{F}(X^{1,0})$ acts trivially on itself, with $(\Ad f)\, g = f[1]\,
g$.

Together these determine the adjoint action of $\Lambda _{F}(X^{1,0})$
on the whole algebra $\Ecal $.  By symmetry, the same holds for 
the adjoint action of any $\Lambda _{F}(X^{m,n})$.

\subsubsection{Anti-involution}
\label{sss:anti-involution}

One can check from the defining relations above that $\Ecal $ has a
further symmetry given by an involutory anti-automorphism (product
reversing automorphism)
\begin{equation}\label{e:involution-Phi}
\begin{gathered}
\Phi \colon \Ecal \rightarrow \Ecal \\
\Phi (g(c_{1},c_{2})) = g(c_{2}^{-1},c_{1}^{-1}),\quad \Phi
(f(X^{m,n}_{\theta })) = f(X^{n,m}_{\pi /2-\theta }).
\end{gathered}
\end{equation}
Note that $\Phi $ is compatible with reflecting degrees in $\ZZ ^{2}$
about the line $x = y$.  Together with $\widehat{\SL _{2}(\ZZ )}$ it
generates a $\widehat{\GL _{2}(\ZZ )}$ action on $\Ecal $ for which
$\rho \in \widehat{\GL _{2}(\ZZ )}$ is an anti-automorphism if
$\widehat{\GL _{2}(\ZZ )}\rightarrow \GL _{2}(\ZZ )\overset{\det
}{\rightarrow }\{\pm 1 \}$ sends $\rho $ to $-1$.

\subsection{Action of \texorpdfstring{$\Ecal $}{E} on
\texorpdfstring{$\Lambda $}{\textLambda}}
\label{ss:E-action}

We write $f^{\bullet }$ for the operator of multiplication by a
function $f$ to better distinguish between operator expressions such
as $(\omega f)^{\bullet }$ and $\omega \cdot f^{\bullet }$.  For $f$ a
symmetric function, $f^{\perp }$ denotes the $\langle -, - \rangle$
adjoint of $f^{\bullet }$.

Here and again later on, we use an overbar to indicate inverting the
variables in any expression; for example
\begin{equation}\label{e:overbar-example}
\overline{M} = (1-q^{-1})(1-t^{-1}).
\end{equation}
We extend the notation in \eqref{e:eigen-operators} accordingly, setting
\begin{equation}\label{e:f(B-bar)}
f[\overline{B}]\, \Htild _{\mu } = f[B_{\mu }(q^{-1},t^{-1})]\, \Htild
_{\mu }\,.
\end{equation}

\begin{prop}[{\cite[Prop 3.3.1]{paths}}]
\label{prop:E-action}
There is an action of $\Ecal $ on $\Lambda $ characterized as follows.

\noindent
(i) The central parameters $c_{1},c_{2}$ act as scalars
\begin{equation}\label{e:center-action}
c_{1}\mapsto 1,\quad
c_{2}\mapsto (q\, t)^{-1}.
\end{equation}
(ii) The subalgebras $\Lambda _{\kk }(X^{\pm 1,0})$ act as
\begin{equation}\label{e:X10-action}
f(X^{1,0})\mapsto (\omega
f)[B-1/M],\quad f(X^{-1,0})\mapsto (\omega f)[\overline{1/M-B}] .
\end{equation}
(iii) The subalgebras $\Lambda _{\kk }(X^{0,\pm 1})$ act as
\begin{equation}\label{e:X01-action}
f(X^{0,1})\mapsto f[-X/M]^{\bullet },\quad f(X^{0,-1})\mapsto
f(X)^{\perp }.
\end{equation}
\end{prop}

We will make particular use of operators representing the action on
$\Lambda $ of elements $p_{1}(X^{a,1})$ and $p_{1}(X^{1,a})$ in $\Ecal
$.  For the first we need the operator $\nabla$, defined
in~\cite{BeGaHaTe99} as an eigenoperator on the modified Macdonald
basis by
\begin{equation}\label{e:nabla}
\nabla \Htild _{\mu } = t^{n(\mu )}q^{n(\mu ^{*})} \Htild _{\mu },
\end{equation}
where $n(\mu )$ is given by \eqref{e:n-lambda} and $\mu ^{*}$ denotes
the transpose partition.

For the second, we introduce the doubly infinite generating series
\begin{equation}\label{e:Dz}
D(z) = \omega \Omega [z^{-1}X] ^{\bullet }(\omega \Omega [-z M
X])^{\perp }\,,
\end{equation}
where $\Omega (X)$ is given by \eqref{e:Omega}.

\begin{defn}
\label{def:Ea-Da}
For $a\in\ZZ$, we define operators on $\Lambda = \Lambda _{\kk }(X)$
\begin{align}
\label{e:Ea}
E_a &= \nabla^a e_1(X)^{\bullet }\,  \nabla^{-a},\\
\label{e:Db} D_a &= \langle z^{-a} \rangle D(z).
\end{align}
The operators $D_a$ are the same as in \cite{paths} and differ by a
sign $(-1)^{a}$ from those in \cite{BeGaHaTe99,GaHaTe99}.
\end{defn}

\begin{prop}
\label{prop:Da-Ea}
In the action of $\Ecal $ on $\Lambda $ given by Proposition
\ref{prop:E-action}:

\noindent
(i) the element $p_{1}[-M X^{1,a}] = -M p_{1}(X^{1,a})\in \Ecal $ acts
as the operator $D_{a}$;

\noindent 
(ii) the element $p_{1}[-M X^{a,1}] = -M p_{1}(X^{a,1})\in \Ecal $
acts as the operator $E_a$.
\end{prop}

\begin{proof}
Part (i) is proven in \cite[Prop 3.3.4]{paths}.
 
By \eqref{e:X01-action}, $p_1[-MX^{0,1}]$ acts on $\Lambda$ as
multiplication by $p_1[X] = e_1(X)$.  It was shown in
\cite[Lemma 3.4.1]{paths} that the action of $\Ecal $ on
$\Lambda $ satisfies the symmetry $\nabla f(X^{m,n}) \nabla^{-1} =
f(X^{m+n,n})$.  More generally, this implies $\nabla^{a} f(X^{m,n})
\nabla^{-a} = f(X^{m+an,n})$ for every integer $a$.  Hence,
$p_{1}[-MX^{a,1}]$ acts as $\nabla^{a} p_{1}[-M X^{0,1}] \nabla^{-a} =
\nabla^{a} e_{1}(X)^{\bullet }\, \nabla^{-a}$.
\end{proof}

\subsection{\texorpdfstring{$\GL_{l}$}{GL\_l} characters and the
shuffle algebra}
\label{ss:GL-characters}

As usual, the weight lattice of $\GL _{l}$ is $\ZZ ^{l}$, with Weyl
group $W = S_{l}$ permuting the coordinates.  A weight $\lambda $ is
dominant if $\lambda _{1}\geq \cdots \geq \lambda _{l}$.  A {\em
polynomial weight} is a dominant weight $\lambda $ such that $\lambda
_{l}\geq 0$.  In other words, polynomial weights of $\GL _{l}$ are
integer partitions of length at most $l$.

The algebra of virtual $\GL _{l}$ characters over $\kk $ can be
identified with the algebra of symmetric Laurent polynomials $\kk
[x_{1}^{\pm 1},\ldots,x_{l}^{\pm 1}]^{S_{l}}$.  If $\lambda $ is a
polynomial weight, the irreducible character $\chi _{\lambda }$ is
equal to the Schur function $s_{\lambda }(x_{1},\ldots,x_{l})$.  Given
a virtual $\GL _{l}$ character $f(x)= f(x_1,\dots,x_l) = \sum
_{\lambda }c_{\lambda }\chi _{\lambda }$, the partial sum over
polynomial weights $\lambda $ is a symmetric polynomial in $l$
variables, which we denote by $f(x)_{\pol }$.  We use the same
notation for infinite formal sums $f(x)$ of irreducible $\GL _{l}$
characters, in which case $f(x)_{\pol }$ is a symmetric formal power
series.

The Weyl symmetrization operator for $\GL _{l}$ is
\begin{equation}\label{e:Weyl-symmetrization}
\sigmabold (\phi (x_{1},\ldots,x_{l})) = \sum _{w\in S_{l}}
w\left(\frac{\phi (x)}{ \prod _{i<j} (1-x_{j}/x_{i})} \right).
\end{equation}
For dominant weights $\lambda$, the Weyl character formula can be
written $\chi _{\lambda } = \sigmabold (x^{\lambda })$.  More
generally, if $\phi (x) = \phi (x_{1},\ldots,x_{l})$ is a Laurent
polynomial over any field $\kk $, then $\sigmabold (\phi (x))$ is a
virtual $\GL_{l}$ character over $\kk $.

The Hall-Littlewood symmetrization operator is defined by
\begin{equation}\label{e:Hq}
\Hbold ^{l}_q(\phi (x)) = \sigmabold \left( \frac{\phi (x)}{ \prod
_{i<j}(1-q\, x_{i}/x_{j})} \right).
\end{equation}
If $\phi (x) = \phi (x_{1},\ldots,x_{l})$ is a rational function over
a field $\kk $ containing $\QQ (q)$, then $\Hbold ^{l}_{q}(\phi (x))$
is a symmetric rational function over $\kk $.  If $\phi (x)$ is a
Laurent polynomial, we can also regard $\Hbold ^{l}_{q}(\phi (x))$ as
an infinite formal sum of $\GL _{l}$ characters with coefficients in
$\kk $, by interpreting the factors $1/(1-q\, x_{i}/x_{j})$ as
geometric series.  We always understand $\Hbold ^{l}_{q}(\phi (x))$ in
this {\em raising operator series} sense when taking the polynomial
part $\Hbold ^{l}_{q}(\phi (x))_{\pol }$.

We also use the two-parameter symmetrization operator
\begin{equation}\label{e:Hqt}
\Hbold^{l}_{q,t}(\phi(x)) = \Hbold^l_q\left(\phi(x) \prod_{i<j}\frac{
(1-q\,t\,x_{i}/x_{j})} {(1-t\, x_{i}/x_{j})} \right) = \sigmabold
\left(\frac{\phi (x)\prod _{i<j}(1-q\, t\, x_{i}/x_{j})}{\prod
_{i<j}\bigl((1-q\, x_{i}/x_{j})(1-t\, x_{i}/x_{j})\bigr)} \right).
\end{equation}
Again, if $\phi (x)$ is a rational function over $\kk =\QQ (q,t)$,
then $\Hbold^{l}_{q,t}(\phi(x))$ is a symmetric rational function over
$\kk $, while if $\phi (x)$ is a Laurent polynomial, or more generally
a rational function which has a power series expansion in the
$x_{i}/x_{j}$ for $i<j$, we can also interpret
$\Hbold^{l}_{q,t}(\phi(x))$ as a raising operator series.  The series
interpretation always applies when taking
$\Hbold^{l}_{q,t}(\phi(x))_{\pol }$.

Fixing $\kk =\QQ (q,t)$ once again, let $T = T(\kk[z^{\pm 1}])$ be the
tensor algebra on the Laurent polynomial ring in one variable, that
is, the non-commutative polynomial algebra with generators
corresponding to the basis elements $z^{a}$ of $\kk [z^{\pm 1}]$ as a
vector space.  Identifying $T^{m} = T^{m}(\kk [z^{\pm 1}])$ with
$\kk[z_{1}^{\pm 1},\ldots,z_{m}^{\pm 1}]$, the product in $T$ is given
by `concatenation,'
\begin{equation}\label{e:concatenation-product}
f\cdot g = f(z_{1},\ldots,z_{k})g(z_{k+1},\ldots,z_{k+l}),\quad
\text{for $f\in T^{k}$, $g\in T^{l}$}.
\end{equation}
The Feigin-Tsymbauliak {\em shuffle algebra} \cite{FeigTsym11} is the
quotient $S=T/I$, where $I$ is the graded two-sided ideal whose degree
$l$ component $I^{l}\subseteq T^{l}$ is the kernel of the
symmetrization operator $\Hbold ^{l}_{q,t}$ in variables
$z_{1},\ldots,z_{l}$.

Let $\Ecal ^{+} \subseteq \Ecal$ be the subalgebra generated by the $\Lambda
_{\kk }(X^{m,n})$ for $m>0$.  We leave out the central subalgebra $F$,
since the relations of $\Ecal ^{+}$ (as we will see in a moment) do not
depend on the central parameters.

The image of $\Ecal ^{+}$ under the anti-automorphism $\Phi $ in \S
\ref{sss:anti-involution} is the subalgebra $\Phi (\Ecal ^{+})$
generated by the $\Lambda _{\kk }(X^{m,n})$ for $n>0$.  Note that our
convention $\theta \in (-\pi ,\pi ]$ when the subscript is omitted
yields $\Phi (f(X^{m,n})) = f(X^{n,m})$
for $\Lambda _{\kk
}(X^{m,n})\subseteq \Ecal ^{+}$, since the branch cut is in the third
quadrant.

Schiffmann and Vasserot \cite{SchiVass13} proved the following result.
See \cite[\S 3.5]{paths} for more details on the translation of their
theorem into our notation.

\begin{prop}[{\cite[Theorem 10.1]{SchiVass13}}]\label{prop:shuffle-isomorphism}
There is an algebra isomorphism $\psi \colon S\rightarrow \Ecal ^{+}$
and an anti-isomorphism $\psi^{\op } = \Phi \circ \psi \colon
S\rightarrow \Phi (\Ecal ^{+})$, given on the generators by $\psi
(z^{a}) = p_{1}[-MX^{1,a}]$ and $\psi^{\op } (z^{a}) =
p_{1}[-MX^{a,1}]$.
\end{prop}

To be clear, on monomials in $m$ variables, representing elements of
tensor degree $m$ in $S$, the maps in
Proposition~\ref{prop:shuffle-isomorphism} are given by
\begin{align}\label{e:psi-spelled-out}
\psi (z_{1}^{a_{1}}\cdots z_{m}^{a_{m}})& = p_{1}[-M
X^{1,a_{1}}]\cdots p_{1}[-M X^{1,a_{m}}]\\
\label{e:psi-op-spelled-out}
\psi ^{\op }(z_{1}^{a_{1}}\cdots z_{m}^{a_{m}})& = p_{1}[-M
X^{a_{m},1}]\cdots p_{1}[-M X^{a_{1},1}]
\end{align}

Later we will need the following formula for the action of $\psi (\phi
(z))$ on $\Lambda (X)$.

\begin{prop}[{\cite[Proposition 3.5.2]{paths}}]\label{prop:H-formula}
Let $\phi(z) = \phi(z_{1},\ldots,z_{l})$ be a Laurent polynomial
representing an element of tensor degree $l$ in $S$, and let $\zeta =
\psi (\phi (z)) \in \Ecal ^{+}$ be its image under the map in
\eqref{e:psi-spelled-out}.  With $\Ecal $ acting on $\Lambda $ as in
Proposition \ref{prop:E-action}, we have
\begin{equation}\label{e:H-formula}
\omega (\zeta \cdot 1)(x_{1},\ldots,x_{l}) = \Hbold^l _{q,t}(\phi(x))_{\pol }.
\end{equation}
\end{prop}

\section{Schiffmann algebra reformulation of the symmetric function side}
\label{ss:symm-side-reform}

\subsection{Distinguished elements \texorpdfstring{$D_{\bb }$}{D\_b}
and \texorpdfstring{$E_{\aA }$}{E\_a}} Negut~\cite{Negut14} defined a
family of distinguished elements $D_{\bb }\in \Ecal ^{+}$, indexed by
$\bb \in \ZZ ^{l}$, which in the case $l=1$ reduce to the elements in
Proposition~\ref{prop:Da-Ea}(i).  Here a remarkable symmetry between
these elements and their images $E_{\aA }$ under the anti-involution
$\Phi $ will play a crucial role.
After defining the Negut elements, we derive this symmetry in
Proposition~\ref{prop:D=E} with the help of a commutator formula of
Negut \cite{Negut18}.

\begin{defn}[see also {\cite[\S 3.6]{paths}}]\label{def:Negut-Db}
Given $\bb = (b_1,\ldots,b_l) \in \ZZ^l$, set
\begin{equation}\label{e:Negut-phi}
\phi ({z}) = \frac{z_{1}^{b_1}\cdots z_{l}^{b_{l}}}{
\prod _{i=1}^{l-1}(1-q\, t\, z_{i}/z_{i+1})}.
\end{equation}
and let $\nu(z)=\nu (z_{1},\ldots,z_{l})$ be a Laurent polynomial
satisfying $\Hbold_{q,t}^l(\nu(z)) = \Hbold_{q,t}^l(\phi(z))$.  Such a
$\nu (z)$ exists by \cite[Proposition 6.1]{Negut14}, and represents a
well-defined element of the shuffle algebra $S$.  The {\em Negut
element} $D_{\bb }$ and the {\em transposed Negut element} $E_{\aA }$,
where $\aA = (b_{l},\ldots,b_{1})$ is the reversed sequence of
indices, are defined by
\begin{align}
\label{e:Edef}
D_{\bb} = D_{b_{1},\ldots,b_{l}} & = \psi (\nu (z)) \in \Ecal^+\\
E_{\aA } = E_{b_{l},\ldots,b_{1}} & =\Phi ( D_\bb ) = \psi ^{\op }(\nu
(z))\in \Phi(\Ecal ^{+}).
\end{align}
\end{defn}

We should point out that, strictly speaking, the Negut elements in the
case $l=1$ are defined to be elements $D_{a} = p_{1}[-M X^{1,a}]$ and
$E_{a} = p_{1}[-M X^{a,1}]$ of $\Ecal $, while in
Definition~\ref{def:Ea-Da}, we used the notation $D_{a}$ and $E_{a}$
for operators on $\Lambda $.  However, by
Proposition~\ref{prop:Da-Ea}, these Negut elements act as the
operators with the same name, so no confusion should ensue.

Later we will use the following product formulas, which are immediate
from Definition~\ref{def:Negut-Db}.
\begin{eqnarray}
\label{e:product of D} D_{b_{1},\ldots,b_{l}}\, D_{b_{l
+1},\ldots,b_{n}} &=& D_{b_{1},\ldots,b_{n}}
 - q\, t\,  D_{b_{1},\ldots,b_{l} + 1, b_{l +1} - 1,\ldots,b_{n}}\,,\\
\label{e:product of E} E_{a_{n},\ldots,a_{l +1}}\,
E_{a_{l},\ldots,a_{1}} &=& E_{a_{n},\ldots,a_{1}} - q\, t\,
E_{a_{n},\ldots, a_{l +1} - 1,a_{l} + 1,\ldots,a_{1}}\,.
\end{eqnarray}

As noted in \S \ref{sss:remarks}, the internal action relations
determine the action of $\Lambda _{\kk }(X^{0,1})$ on $\Phi( \Ecal
^{+})$.  Using the anti-isomorphism between $\Phi (\Ecal ^{+})$ and
the shuffle algebra we can make this more explicit.

\begin{lemma}\label{lem:Ad-formula}
Let $\phi(z)=\phi (z_{1},\ldots,z_{n})$ be a Laurent polynomial representing
an element of tensor degree $n$ in $S$.  Then
\begin{equation}\label{e:Ad-formula}
(\Ad f(X^{1,0}))\, \psi ^{\op }(\phi (z)) = \psi ^{\op }\bigl((\omega
f)(z_{1},\ldots,z_{n})\cdot \phi (z)\bigr).
\end{equation}
As a particular consequence, we have
\begin{equation}\label{e:adE}
(\Ad f(X^{1,0})) E_{a_{n},\ldots,a_{1}}= \psi^{\op}\left(
\frac{(\omega f)(z_{1},\ldots,z_{n})\cdot z_{1}^{a_{1}}\cdots
z_{n}^{a_{n}}}{\prod_{i=1}^{n-1} (1-q\, t\, z_{i} / z_{i+1})}\right).
\end{equation}
\end{lemma}

\begin{proof}
This follows immediately from the rule in \S \ref{sss:remarks} for
$\Ad f$ acting on a product.
\end{proof}

\subsection{Commutator identity}
\label{ss:commutator-identity}

We use a formula for the commutator of elements $D_{a}$ and
$D_{\bb }$, and a similar identity for $E_{a}$ and $E_{\bb }$.  This
commutation relation was proved geometrically by Negut
in~\cite{Negut18}, but to keep things self-contained, we provide an
elementary algebraic proof. It is convenient to express the formula
using the notation
\begin{equation}
\label{e:comnot}
\sump{i=a}{b} f_i =
\begin{cases}
\sum_{i = a}^{b} f_i & \text{for $a \leq b + 1$} \\[2ex]
- \sum_{i = b+1}^{a-1} f_i & \text{for $a \geq b + 1$}.
\end{cases}
\end{equation}
As a mnemonic device, note that both cases can be interpreted as
$\sum_{i = a}^{\infty} f_i -  \sum_{i = b+1}^{\infty} f_i$.

\begin{prop}[{\cite[Proposition 4.7]{Negut18}}]\label{prop:commutator-identity}
For any $a\in\ZZ$ and $\mathbf{b} = (b_1,\ldots,b_l) \in \ZZ^l$, we have
\begin{align}\label{e:D-commutator}
[D_a,D_{b_1,b_2,\ldots,b_l}]
& = -M\sum _{i = 1}^l\sump{k = a+1}{b_i} D_{b_1,\ldots,b_{i-1}, k,
 a+b_i-k,b_{i+1},\ldots,b_l}\\
\label{e:E-commutator}
[E_{b_l, \ldots, b_2 , b_1},E_a] & =
- M\sum _{i = 1}^l\sump{k = a+1}{b_i} E_{b_l,\ldots,b_{i+1},
a+b_i-k,k,b_{i-1},\ldots,b_1}.
\end{align}
\end{prop}

We will need the following lemma for the proof.  The notation $\Omega
(X)$ is defined in~\eqref{e:Omega}.  Since plethystic substitution into
$\Omega (X)$ is characterized by
\begin{equation}\label{e:Omega-pleth}
\Omega [a_{1}+a_{2}+\cdots -b_{1}-b_{2}-\cdots ] = \frac{\prod
_{i}(1-b_{i})}{\prod _{i}(1-a_{i})},
\end{equation}
we have
\begin{equation}\label{eq:Omega-Mz}
\Omega[Mz] = \frac{(1-q\, z)(1-t\, z)}{(1-z) (1-q\, t\, z)} \quad
\text{and}\quad \Omega[-Mz] = \frac{(1-z) (1-q\, t\, z)}{(1-q\,z)(1-t\, z)}\, .
\end{equation}

\begin{lemma}\label{l:antisym}
For any $f(z)=f(z_1,\ldots,z_m)$ antisymmetric in $z_i$ and $z_{i+1}$,
we have
\begin{equation}
\Hbold^m_{q,t}\bigl(\Omega[M\, z_i/z_{i+1}] f(z)\bigr) = 0 \, .
\end{equation}
\end{lemma}

\begin{proof}
The definition of $\Hbold^m_{q,t}$ and~\eqref{eq:Omega-Mz} imply that
\begin{align}
\Hbold^m_{q,t}\bigl(\Omega[M\, z_i/z_{i+1}]f(z)\bigr) = \sum_{w\in
S_m} w\left( f(z)\prod_{j \neq k} \frac{1}{1-z_j/z_k}\prod_
{\substack{j<k \\ (j,k)\neq (i,i+1)}} \Omega[-M\, z_j/z_{k}] \right),
\end{align}
which vanishes since $f(z)$ is antisymmetric in $z_i$ and $z_{i+1}$.
\end{proof}

\begin{proof}[Proof of Proposition \ref{prop:commutator-identity}]

Identity \eqref{e:E-commutator} for $[E_{b_{l},\ldots,b_{1}},\,
E_{a}]$ follows from \eqref{e:D-commutator} by applying the
anti-homomorphism $\Phi $, so we only prove \eqref{e:D-commutator},
which can be written
\begin{equation}\label{e:D-commutator-0}
D_a\, D_\bb - D_\bb\, D_a + M \sum_{i=1}^{l}\,  \sump{k=a+1}{b_{i}}
D_{b_1,\ldots,b_{i-1}, k, a+b_i-k,b_{i+1},\ldots,b_l} = 0.
\end{equation}
Using Definition~\ref{def:Negut-Db} and the isomorphism $\psi \colon
S\rightarrow \Ecal ^{+}$, we can prove \eqref{e:D-commutator-0} by
showing that a rational function representing the left hand side is in
the kernel of the symmetrization operator $\Hbold ^{l+1}_{q,t}$.
For this we can work directly with the rational functions $\phi (z)$
in~\eqref{e:Negut-phi}; there is no need to replace them explicitly
with Laurent polynomials having the same symmetrization.

Let $\phi(z)$ be the function in~\eqref{e:Negut-phi} for $D_{\bb }$,
and set
\begin{equation}\label{eq:phihatz}
\phi (\hat z_i) = \phi(z_1,\ldots,z_{i-1},z_{i+1},\ldots,z_{l+1}) =
\dfrac{z_{1}^{ b_1}\cdots z_{i-1}^{b_{i-1}}z_{i+1}^{b_{i}}\cdots
z_{l+1}^{b_{l} }}{ (1-q\, t\, z_{i-1}/z_{i+1})\prod\limits _
{\substack{1\leq j\leq l\\ j\neq i,i-1}} (1-q\, t\, z_{j}/z_{j+1})}\,.
\end{equation}
To prove \eqref{e:D-commutator-0}, we want to show
\begin{equation}\label{e:aimforD}
\Hbold_{q,t}^{l+1}\bigg( z_1^a\phi(\hat{z_1 })
-\phi(\hat{z_{l+1}})z_{l+1}^a + M \dfrac{\sum\limits _{i=1}^{l}\,
\sump{k=a+1}{b_{i}} z_1^{b_1}\cdots z_{i-1}^{b_{i-1}} z_i^k
z_{i+1}^{a+b_i - k}z_{i+2}^{b_{i+1}}\cdots
z_{l+1}^{b_l}}{\prod_{j=1}^{l} (1-q\,t\,z_j/z_{j+1})} \bigg)=0\,.
\end{equation}

Since $z_i^a\phi(\hat z_i) -\phi(\hat z_{i+1})z_{i+1}^a$ is
antisymmetric in $z_i$ and $z_{i+1}$, Lemma~\ref{l:antisym} implies
\begin{equation}\label{e:H_q,t(omega f) =0}
\sum _{i=1}^{l} \Hbold_{q,t}^{l+1} \biggl( \Omega [M\, z_{i}/z_{i+1}]
(z_i^a\phi(\hat z_i) -\phi(\hat z_{i+1})z_{i+1}^a) \biggr) = 0
\end{equation}
The first formula in \eqref{eq:Omega-Mz} is algebraically the same as
\[
\Omega[M\, z] = 1-\dfrac{M}{(1-z^{-1})(1-q\, t\, z)}.\]
After substituting this into \eqref{e:H_q,t(omega f) =0}, the
linearity of $\Hbold_{q,t}^{l+1}$ gives
\begin{equation}\label{e:H_q,t(omega f) =0 rewritten}
\Hbold_{q,t}^{l+1}\bigg(\sum_{i = 1}^{l} \Big(z_i^a\phi(\hat{z_i})
-\phi(\hat{z_{i+1}})z_{i+1}^a -M \frac{z_i^a\phi(\hat{z_i})
-\phi(\hat{z_{i+1}})z_{i+1}^a}{(1-z_{i+1}/z_i)(1-q\, t\,
z_i/z_{i+1})}\Big) \bigg)=0.
\end{equation}
The terms $z_i^a\phi(\hat{z_i}) -\phi(\hat{z_{i+1}})z_{i+1}^a$
telescope, reducing this to
\begin{equation}\label{e:H_q,t(omega f) =0 simplify}
\Hbold_{q,t}^{l+1}\bigg( z_1^a\phi(\hat{z_1 })
-\phi(\hat{z_{l+1}})z_{l+1}^a -M \sum_{i = 1}^{l}
\frac{z_i^a\phi(\hat{z_i})
-\phi(\hat{z_{i+1}})z_{i+1}^a}{(1-z_{i+1}/z_i)(1-q\, t\,
z_i/z_{i+1})}\bigg)=0.
\end{equation} 

If we use the convention $z_0=0$ and $z_{l+2}=\infty$, collecting
terms in $z_{i}^{a}\phi (\hat{z_{i}})$ and some further algebra manipulations
give
\begin{align*}
\sum _{i=1}^{l} \frac{z_i^a \phi(\hat{z_i}) - \phi(\hat{z_{i+1}})
z_{i+1}^a}{(1 - \frac{z_{i+1}}{z_{i}})(1 - q\, t
\frac{z_{i}}{z_{i+1}})} & = \sum _{i=1}^{l+1} \left[\frac{1}{(1 -
\frac{z_{i+1}}{z_{i}})(1 - q\, t\frac{z_{i}}{z_{i+1}})} - \frac{1}{(1
- \frac{z_{i}}{z_{i-1}})(1 -
q\, t\frac{z_{i-1}}{z_{i}})}\right] z_i^a \phi(\hat{z_i}) \\
& = \sum _{i=1}^{l+1} \frac{z_i^a \phi(\hat{z_i}) (1 - q\, t
\frac{z_{i-1}}{z_{i+1}})}{(1 - q\, t \frac{z_{i-1}}{z_{i}})(1 - q\, t
\frac{z_{i}}{z_{i+1}})} \Big( \frac{1}{1 -
\frac{z_{i+1}}{z_{i}}} - \frac{1}{1 - \frac{z_{i}}{z_{i-1}}}\Big)\\
& = \sum _{i=1}^{l+1} \frac{\dfrac{z_i^a \phi(\hat{z_i}) (1 - q\, t
\frac{z_{i-1}}{z_{i+1}})}{(1 - q\, t \frac{z_{i-1}}{z_{i}})(1 - q\, t
\frac{z_{i}}{z_{i+1}})} - \dfrac{z_{i+1}^a \phi(\hat{z_{i+1}}) (1 - q\,
t \frac{z_{i}}{z_{i+2}})}{(1 - q\, t \frac{z_{i}}{z_{i+1}})(1 - q\, t
\frac{z_{i+1}}{z_{i+2}})}}{1 - \frac{z_{i+1}}{z_i}}\,.
\end{align*}
Expanding the definition~\eqref{eq:phihatz} of $\phi (\hat{z_{i}})$
for each $i$ yields
\begin{equation*}
\frac{z_i^a \phi(\hat{z_i}) (1 -q\, t\, z_{i-1}/z_{i+1})}{(1 -q\, t\,
z_{i-1} / z_{i}) (1- q\, t\, z_{i}/z_{i+1})} = \frac{z_1^{b_1}\cdots
z_{i-1}^{b_{i-1}} z_i^a z_{i+1}^{b_i} \cdots
z_{l+1}^{b_l}}{\prod_{j=1}^{l} (1 - q\, t\, z_j / z_{j+1})} \,,
\end{equation*}
so that
\begin{align*}
\sum _{i=1}^{l} \frac{z_i^a \phi(\hat{z_i}) - \phi(\hat{z_{i+1}})
z_{i+1}^a}{(1 - z_{i+1} / z_{i})(1 - q\, t\, z_{i} / z_{i+1})}
& = \frac{\sum _{i=1}^{l} z_1^{b_1} \cdots z_{i-1}^{b_{i-1}} \cdot
\dfrac{z_i^a z_{i+1}^{b_i} -z_i^{b_i} z_{i+1}^{a}}{1 - z_{i+1} / z_i}
\cdot z_{i+2}^{b_{i+1}}\cdots z_{l+1}^{b_l}}{\prod _{j=1}^{l} (1 - q\,
t\, z_j / z_{j+1})}\\
& = \frac{- \sum _{i=1}^{l} z_1^{b_1} \cdots z_{i-1}^{b_{i-1}} \cdot
\bigl(\,  \sump{k = a+1}{b_{i}} z_i^kz_{i+1}^{a+b_i - k}  \bigr)
\cdot z_{i+2}^{b_{i+1}}\cdots z_{l+1}^{b_l}}{\prod _{j=1}^{l} (1 - q\,
t\, z_j / z_{j+1})}
\end{align*}
Identity ~\eqref{e:aimforD} follows by substituting this back into
\eqref{e:H_q,t(omega f) =0 simplify}.
\end{proof}

\subsection{Symmetry identity for \texorpdfstring{$D_{\bb }$}{D\_b}
and \texorpdfstring{$E_{\aA }$}{E\_a}}
\label{ss:Db=Ea}

Next we will prove an identity between certain instances of the Negut
elements $D_{\bb }\in \Ecal ^{+}$ and transposed Negut elements
$E_{\aA }\in \Phi (\Ecal ^{+})$.  Before stating the identity we need
to describe how the indices $\aA $ and $\bb $ will correspond.
 
\begin{defn}
A south-east lattice path $\gamma$ from $(0,n)$ to $(m,0)$, for
positive integers $m, n$, is \emph{admissible} if it starts with a
south step and ends with an east step; that is, $\gamma$ has a step
from $(0,n)$ to $(0,n-1)$ and one from $(m-1,0)$ to $(m,0)$.  Define
$\mathbf{b}(\gamma) = (b_1, \dots, b_m)$ by taking $b_i = (
\text{vertical run of $\gamma$ at $x = i-1$} )$ for $i = 1,\ldots, m$
and $\mathbf{a}(\gamma)=(a_n, \dots, a_1)$ with $a_j = (
\text{horizontal run of $\gamma$ at $y = j -1$} ) $ for $j = 1,\ldots,
n$.  Set $D_\gamma= D_{\mathbf{b}(\gamma)}$ and $E_\gamma =
E_{\mathbf{a}(\gamma)}$.
\end{defn}

Note that if $\gamma ^{*}$ is the transpose of an admissible path
$\gamma $ with $\bb (\gamma) = (b_1, \ldots, b_m)$ and $\aA (\gamma )
= (a_n, \ldots, a_1)$, as above, then $\aA (\gamma ^{*}) =
(b_m,\dots,b_1) $ and $\bb (\gamma ^{*}) = (a_1,\dots,a_n)$, and
$E_\gamma =\Phi (D_{\gamma^*})$.

\begin{example}\label{example of an admissible lattice path}
Paths $\gamma$ and $\gamma^*$ below are both admissible.  $\gamma$ is
from $(0,8)$ to $(4,0)$ with $\bb(\gamma)=(2,1,3,2)$ and
$\aA(\gamma)=(0,1,1,0,0,1,0,1)$, whereas $\gamma^*$ is from $(0,4)$ to
$(8,0)$ and has $\aA(\gamma^*)=(2,3,1,2)$ and
$\bb(\gamma^*)=(1,0,1,0,0,1,1,0)$.
\[
\begin{tikzpicture}[xscale = 0.5,yscale = 0.5]
	\draw[step=1cm,gray!20,very thin] (0,0) grid (4,8);
	\draw[thick] (0,8)--(0,6)--(1,6)--(1,5)--(2,5)--(2,2)--(3,2)--(3,0)--(4,0);
	\node[below left]  at (1,4) {$\gamma$};
	\draw[step=1cm,gray!20,very thin] (7,0) grid (15,4);
	\draw[thick] (7,4)--(7,3)--(8,3)--(9,3)--(9,2)--(10,2)--(11,2)--(12,2)--(12,1)--(13,1)--(13,0)--(14,0)--(15,0);
	\node[below left]  at (8,1) {$\gamma^*$};
\end{tikzpicture}
\]
\end{example}

\begin{prop}\label{prop:D=E}
For every admissible path $\gamma$ we have $D_\gamma = E_\gamma $.
\end{prop}

\begin{proof}
Let $\gamma $ be an admissible path $\gamma$ from $(0,n)$ to $(m,0)$,
where $m,n$ are positive integers.

We first establish the case when $n = 1$.  In this case, $E_{\gamma }
= E_{m} = p_{1}[-M X^{m,1}]$ and $D_{\gamma } = D_{10^{m-1}}$.  If
$m=1$, these are $E_{1} = D_{1} = p_{1}[-M X^{1,1}]$.  In general,
\eqref{e:internal-action} implies $E_{m}= p_1[-M X^{m,1}] = (\Ad
p_1(X^{1,0}))^{m-1} p_1[-M X^{1,1}] = (\Ad p_1(X^{1,0}))^{m-1} D_{1}$,
while \eqref{e:Ad-examples} and the commutator identity
\eqref{e:D-commutator} imply $(\Ad p_1(X^{1,0})) D_{10^k} =
[p_1(X^{1,0}),\, D_{10^k}] = -(1/M)[D_0, D_{10^k}] = D_{10^{k+1}}$,
and therefore $(\Ad p_{1}(X^{1,0}))^{m-1} D_{1} = D_{10^{m-1}}$.

Using the involution $\Phi $, we can deduce the $m=1$ case from the
$n=1$ case:
\begin{equation}\label{e:Dn=D}
D_\gamma = D_n = \Phi (E_n) = \Phi (D_{1,0^{n-1}}) = E_{0^{n-1},1} =
E_\gamma.
\end{equation}

For $m, n > 1$, we proceed by induction, assuming that the result
holds for paths from $(0,n')$ to $(m',0)$ when $m'\leq m$ and $n'\leq
n$ and $(m',n')\not =(m,n)$.

For a given $m,n$, there are finitely many admissible paths $\gamma $,
and thus a finite dimensional space $V$ of linear combinations $\sum
_{\gamma } c_{\gamma } D_{\gamma }$ involving these paths.
Let $V'\subseteq V$ denote the subspace consisting of linear combinations
which form the left hand side of a valid instance of the identity
\begin{equation}\label{e:identity-space}
\sum
_{\gamma } c_{\gamma } D_{\gamma } = \sum
_{\gamma } c_{\gamma } E_{\gamma }.
\end{equation}
Note that $V'=V$ if and only if 
$D_\gamma = E_\gamma$ for all the paths $\gamma $ in question.

We will use the induction hypothesis to construct enough instances of 
\eqref{e:identity-space} to reduce each $D_\gamma$ modulo $V'$
to  a scalar multiple of
$D_{\gamma _{0}}$, where $\gamma_0$ is the path with a south run from
$(0,n)$ to $(0,0)$ followed by an east run to $(m,0)$.  We will then
prove one more instance of \eqref{e:identity-space} for which the left
hand side reduces 
to a non-zero scalar multiple of $D_{\gamma _{0}}$, showing that $V' = V$.

Suppose now that $\gamma \not =\gamma _{0}$.  Then $\gamma$ contains
an east step from $(m_1-1, n_2)$ to $(m_1, n_2)$ and a south step from
$(m_1, n_2)$ to $(m_1, n_2-1)$ for some $m_1+m_2=m$ and $n_1+n_2=n$.
In particular, $\gamma = \nu \cdot \eta$ for shorter admissible paths
$\nu$ and $\eta$, where $\nu \cdot \eta$ is defined to be the lattice
path obtained by placing $\nu$ and $\eta$ end to end; thus $\nu \cdot
\eta$ traces a copy of $\nu$ from $(0,n_1 + n_2)$ to $(m_1,n_2) $ and
then traces a copy of $\eta$ from $(m_1, n_2)$ to $(m_1 + m_2,0)$.

Define $\gamma'=\nu \cdot' \eta$ to be the admissible path obtained
from $\nu\cdot \eta$ by replacing the east-south corner at $(m_1,n_2)$
with a south-east corner at $(m_1-1,n_2 -1)$; $\gamma'$ contains a
south step from $(m_1-1, n_2)$ to $(m_1-1, n_2-1)$ and an east step
from $(m_1 -1, n_2 -1)$ to $(m_1 , n_2-1)$.

The product formulas $\eqref{e:product of D}$ and $\eqref{e:product of
E}$ imply that the elements corresponding to the paths constructed in
this way satisfy
\begin{equation} \label{e: path product of D}
D_\nu D_\eta = D_{\nu\cdot\eta} - q\, t\, D_{\nu\cdot'\eta}
\qquad\text{and}\qquad E_\nu E_\eta = E_{\nu\cdot\eta} - q\, t\,
E_{\nu\cdot'\eta} \,.
\end{equation}
By induction, $D_\nu=E_\nu$ and $D_\eta =E_\eta$, so \eqref{e: path
product of D} implies $D_{\gamma } - q\, t\, D_{\gamma '} = E_{\gamma
} - q\, t\, E_{\gamma '}$.  In other words, in terms of the space $V'$
defined above, we have $D_{\gamma } \equiv q\, t\, D_{\gamma '}
\pmod{V'}$.  Using this repeatedly, we obtain $D_{\gamma }\equiv (q\,
t)^{h(\gamma )} D_{\gamma _{0}}\pmod{V'}$ for every path $\gamma $,
where $h(\gamma )$ is the area enclosed by the path $\gamma $ and the
$x$ and $y$ axes.

To complete the proof it suffices to establish one more identity of
the form \eqref{e:identity-space}, for which the congruences
$D_{\gamma }\equiv (q\, t)^{h(\gamma )} D_{\gamma _{0}}\pmod{V'}$
reduce the left hand side to a non-zero scalar multiple of $D_{\gamma
_{0}}$.

We can assume by induction that $D_{n,0^{m-2}} = E_{0^{n-1},m-1}$,
since this case has the same $n$ and a smaller $m$.  Taking the
commutator with $p_{1}(X^{1,0})$ on both sides gives
\begin{equation}\label{induction step of D(gamma_0) = E(gamma_0)}
-\frac{1}{M}[D_{0},\, D_{n,0^{m-2}}] = [p_1(X^{1,0}),\, D_{n,0^{m-2}}]
= (\Ad p_1(X^{1,0})) E_{0^{n-1},m-1}.
\end{equation}
Using~\eqref{e:D-commutator} on the left hand side and~\eqref{e:adE}
on the right hand side gives
\begin{equation}\label{e: sum D = sum E}
\sum\limits_{k=0}^{n-1}D_{(n-k,k,0^{m-2})} = \sum\limits_{k =
0}^{n-1} E_{(0^{n-1},m-1)+\varepsilon_{n-k}}.
\end{equation}
Now, for $1\leq k\leq n-1$, we have $D_{(n-k,k,0^{m-2})} = D_{\gamma
}$ and $E_{(0^{n-1}, m-1)+\varepsilon_{n-k}} = E_{\gamma }$ for an
admissible path with $h(\gamma ) = k$, as displayed below.
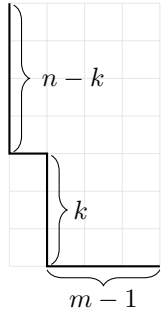
\[
\begin{tikzpicture}[xscale = 0.5,yscale = 0.5]
	\draw[step=1cm,gray!20,very thin] (0,0) grid (4,7);
	\draw[thick] (0,7)--(0,3)--(1,3)--(1,0)--(4,0);
	\draw [decorate,decoration={brace,amplitude= 6pt,mirror,raise=3pt},yshift=0pt]
	(-0.1,3.05) -- (-0.1, 6.95) node [black,midway,xshift= 0.9 cm] {\footnotesize $n-k$};
	\draw [decorate,decoration={brace,amplitude= 6pt,mirror,raise=3pt},yshift=0pt]
	(0.9,0.05) -- (0.9, 2.95) node [black,midway,xshift= 0.5 cm] {\footnotesize $k$};
	\draw [decorate,decoration={brace,amplitude= 6pt,mirror,raise=3pt},yshift=0.5pt]
	(1,0.1) -- (4, 0.1) node [black,midway,yshift= -0.5 cm] {\footnotesize $m-1$};
\end{tikzpicture}
\]
This shows that \eqref{e: sum D = sum E} is an instance of
\eqref{e:identity-space}.  The previous congruences reduce the left
hand side of \eqref{e: sum D = sum E} to $(1+q\, t+\cdots +(q\,
t)^{n-1}) D_{\gamma _{0}}$.  Since the coefficient is non-zero, we
have now established a set of instances of \eqref{e:identity-space}
whose left hand sides span $V$.
\end{proof}

\begin{cor}\label{cor:E-last-index}
For any indices $a_1, \dots, a_l$, we have
\begin{equation}\label{e:E-last-index}
E_{a_l, \dots, a_2, a_1}\cdot 1 = E_{a_l, \dots, a_2, 0}\cdot 1.
\end{equation}
\end{cor}

\begin{proof}
To rephrase, we are to show that $E_{a_{l}, \ldots, a_{2},
a_{1}}\cdot 1$ does not depend on $a_{1}$.  The symmetry
$f(X^{m,n})\mapsto f(X^{m+rn,n})$ of $\Phi (\Ecal ^{+})$
sends
$E_{a_{l},\ldots,a_{1}}$ to $E_{a_{l}+r,\ldots,a_{1}+r}$.  By
\cite[Lemma 3.4.1]{paths}, the action of $\Ecal $ on $\Lambda $
satisfies $\nabla ^{r} f(X^{m,n}) \nabla ^{-r} = f(X^{m+rn,n})$, and
since $\nabla (1) = 1$, this gives $\nabla ^{r} E_{a_{l}, \ldots,
a_{2}, a_{1}}\cdot 1 = E_{a_{l}+r, \ldots, a_{2}+r, a_{1}+r}\cdot 1$.
Hence, we can reduce to the case that $a_{i}>0$ for all $i$.

By \cite[Lemma 3.6.2]{paths}, we have that
$D_{b_{1},\ldots,b_{n},0,\ldots,0} \cdot 1$ is independent of the
number of trailing zeroes.  In the case that $b_{i}\geq 0$ for all $i$
and $b_{1}>0$, this and Proposition~\ref{prop:D=E} imply that
$E_{a_{l},\ldots,a_{1}}\cdot 1$ is independent of $a_{1}$, provided
that $a_{i}\geq 0$ for all $i$ and $a_{1}>0$.  However, we already saw
that this suffices.
\end{proof}

\subsection{Shuffling the symmetric function side of the Extended
Delta Conjecture}
\label{ss:shuffling}

We can now give the promised reformulation of \eqref{e:sside}.

\begin{thm}\label{t:ourLHS}
For $0\leq l<m\leq N$, we have
\begin{equation}\label{e:ourLHS}
\bigl( \omega ( h_l[B]e_{m-l-1}[B-1] e_{N-l}) \bigr)
(x_1,\ldots,x_{m}) = \Hbold^{m}_{q,t} \left( \phi(x)\right)_{\rm
pol}\,,
\end{equation}
where
\begin{equation}\label{e:phi-for-our-LHS}
\phi(x)= \frac{x_1\cdots x_{m}}{\prod _{i } (1 - q\, t\, x_i/x_{i+1})}
h_{N-m}(x_1,\ldots,x_{m}) \overline{e_l(x_2,\ldots,x_{m})},
\end{equation}
and $\overline{e_l(x_2,\ldots,x_{m})} =
e_l(x_2^{-1},\ldots,x_{m}^{-1})$ by our convention on the use of the
overbar.
\end{thm}

\begin{proof}
For any symmetric function $f$
set $g(X) = (\omega f)[X+1/M]$; then \eqref{e:X10-action} gives an identity in
$\Lambda $ for every $\zeta \in \Ecal $

\begin{equation}\label{e:f[B]-zeta-1}
f[B]\, \zeta \cdot 1 = g(X^{1,0})\, \zeta \cdot 1 = \sum ((\Ad
g_{(1)}(X^{1,0}))\, \zeta )\, g_{(2)}(X^{1,0}) \cdot 1,
\end{equation}
where $g[X+Y] = \sum g_{(1)}(X)g_{(2)}(Y)$ in Sweedler notation and we
used the general formula $g \,\zeta = \sum ((\Ad g_{(1)})
\zeta)g_{(2)}$.  Since $g[X+Y] = (\omega f)[X+Y+1/M]$, and $h[B]\cdot
1 = h[0]\cdot 1$ for any $h(X)$, the right hand side of
\eqref{e:f[B]-zeta-1} is equal to
\begin{multline}\label{e:f[B]-zeta-1a}
\sum ((\Ad \, (\omega f)_{(1)}(X^{1,0}))\, \zeta )\,
(\omega f)_{(2)}[X^{1,0}+1/M] \cdot 1 \\ 
= \sum ((\Ad \, (\omega f)_{(1)}(X^{1,0}))\, \zeta )\, (\omega
f)_{(2)}[0] \cdot 1 = ((\Ad \, (\omega f)(X^{1,0}))\, \zeta )\cdot 1.
\end{multline}
Let $n=N-l$.  Taking $\zeta = E_{a_{n},\ldots,a_{1}}$ and using
\eqref{e:adE}, this gives
\begin{equation}\label{e:f[B]-Ea-1}
f[B] E_{a_n,\ldots,a_1}\cdot 1 = f(z_{n},\ldots ,z_{1}) \mathbin{\Big|}
{z_n^{r_n}\cdots z_1^{r_1}\mapsto E_{a_n + r_n, \ldots, a_2 + r_2,
a_1+r_{1}}} \cdot 1.
\end{equation} 
By Corollary~\ref{cor:E-last-index}, the right hand side is a function
of $f(z_{n},\ldots,z_{2},1)$, since the substitution for the monomial
$z^{{\bf r}}$ does not depend on the exponent $r_{1}$.  Expressing
$f(z_{n},\ldots,z_{2},1)$ as $f[z_{n}+\cdots +z_{2}+1]$ and then
substituting $f[X-1]$ for $f(X)$ yields
\begin{equation}\label{e:f[B-1]-Ea-1}
f[B-1] E_{a_n,\ldots,a_1}\cdot 1 = f[z_{n} +\cdots +z_{2}] \mathbin{\Big|}
{z_n^{r_n}\cdots z_2^{r_2}\mapsto E_{a_n + r_n, \ldots, a_2 + r_2,
a_1}} \cdot 1.
\end{equation} 

By \cite[Proposition 6.7]{Negut14}, $E_{0^n} = \Phi(D_{0^n}) =
\Phi(e_n[-MX^{1,0}]) = e_n[-MX^{0,1}]$
(see also \cite[Proposition
3.6.1]{paths}).

Using \eqref{e:f[B-1]-Ea-1}, we therefore obtain
\begin{multline}\label{e: e_k-1[B-1]e_n(x) interms of E}
e_{k-1}[B-1]e_n =e_{k-1} [z_n+\cdots+z_2] \mathbin{\Big|} z_n^{r_n}\cdots
z_2^{r_2}\mapsto E_{r_n, \ldots, r_2, 0} \cdot
1\\
=\sum\limits_{|I|=k-1} E_{\varepsilon_I,0}\cdot 1 =
\sum\limits_{|I|=k-1} E_{\varepsilon_I,1}\cdot 1\,,
\end{multline}

where the sum is over subsets $I\subseteq [n-1]$ and $\varepsilon _{I} =
\sum _{i\in I}\varepsilon _{i}$.  The terms in the last sum are just
$E_{\aA(\nu)}\cdot 1$ for paths $\nu$ from $(0,n)$ to $(k,0)$ with
single east steps on any $k-1$ chosen lines $y=j$ for $j\in [n-1]$,
and a final east step at $y = 0$.  Denote the set of these admissible
paths by $\mathcal P_{k,n}$. For instance, with $n=8$ and $k=4$, the
path $\gamma$ in Example~\ref{example of an admissible lattice path}
corresponds to $E_{\gamma } = E_{0,1,1,0,0,1,0,1}$.

By~\eqref{e:f[B]-Ea-1}, applying $h_l[B]$ to
~\eqref{e: e_k-1[B-1]e_n(x) interms of E} gives
\begin{equation}\label{e: h_l[B]e_k-1[B-1]e_n(x) interms of E}
h_l[B]e_{k-1}[B-1]e_n = \sum_{\nu\in\mathcal P_{k,n}} \;
\sum_{\substack{\rr \in\mathbb N^{n} \\ |\rr |=l}}
E_{\rr +\aA(\nu)}\cdot 1 \,.
\end{equation}
This last expression is the sum of $E_{\gamma }\cdot 1$ over admissible paths
$\gamma $ from $(0,n)$ to $(k+l,0)$, together with a choice of $k-1$
indices $j\in [n-1]$ for which $\gamma $ has at least one east step on
the line $y = j$.  We can consider these indices as distinguishing
$k-1$ east-south corners in $\gamma $.  However, we can also distinguish
these corners by their $x$ coordinates, that is, by a set of $k-1$
indices $i\in [k+l-1]$ for which $\gamma $ has at least one south step
on the line $x = i$.  Setting $m = k+l$ and using
Proposition~\ref{prop:D=E}, this yields the identity
\begin{equation}\label{e:hee}
h_l[B] e_{m-l-1}[B-1] e_n =
\sum_{\substack{\sS \in \NN^{m}: |\sS |=n-k \\
I\subseteq [2,m],|I|=l}} D_{\sS + (1^{m})-\varepsilon_I}\cdot 1 \,.
\end{equation}
Now, since
\begin{equation}\label{e:phi(x)-expanded}
\sum _{\substack{\sS \in \NN ^{m}: |\sS|=n-k \\
I\subseteq [2,m],|I|=l}} {x^{\sS + (1^{m}) - \varepsilon_I}} = x_1\, x_2
\cdots x_{m} h_{n-k}(x_1,\ldots,
x_{m})\overline{e_l(x_2,\ldots,x_{m})}\,,
\end{equation}
the definition of $D_{\bb }$ and Proposition~\ref{prop:H-formula}
imply that 
\begin{equation}\label{e:ourLHS-final}
\omega \biggl(\, \,  \sum_{\substack{\sS \in \NN^{m}: |\sS |=n-k \\
I\subseteq [2,m],|I|=l}} D_{\sS + (1^{m})-\varepsilon_I}\cdot
1 \biggr)(x_{1},\ldots,x_{m}) = 
\Hbold ^{m}_{q,t}(\phi (x))_{\pol }
\end{equation}
with $\phi
(x)$ given by \eqref{e:phi-for-our-LHS}.
\end{proof}

\begin{remark}\label{r:sfsfinitesuffices}
For any ${\bf b}\in\ZZ^m$, \cite[Corollary 3.7.2]{paths} gives that
the Schur expansion of $\omega (D_{\bf b}\cdot 1)$ involves only
$s_\lambda(X)$ with $\ell(\lambda)\leq m$. Hence, although
Theorem~\ref{t:ourLHS} is a statement in $m$ variables, it determines
$\omega ( h_l[B]e_{m-l-1}[B-1] e_{N-l})$ by \eqref{e:hee}.
\end{remark}

\section{Reformulation of the combinatorial side}
\label{s:reform-combinatorial}
\subsection{}
We reformulate~\eqref{eq:EE} by explicitly extracting the coefficient of
$z^{N-m}$; the natural result involves a $q$-weighted tableau generating
function $N_{\beta/\alpha}$ rather than  partially labelled paths.  
For now, we  work only with the tableau description of $N_{\beta/\alpha}$, but
in \S\ref{S:LLTseries} we will see that  $N_{\beta/\alpha}$ is a truncation of
LLT series introduced by Grojnowski and Haiman in \cite{GrojHaim07}.

The $q$-weight in our reformulation involves two auxiliary statistics:
for $\eta,\tau\in\NN^m$, define
\begin{equation}
\label{e:detagamma}
d(\eta,\tau) =
\sum_{1 \leq j < r \leq m} \big| [\eta_{j},\eta_{j}+\tau_j]
      \cap [\eta_{r},\eta_{r}+\tau_r-1] \big| \,,
\end{equation}
with $[a,b]=\{a,\ldots,b\}$ and $[b]=[1,b]$, and for a vector $\eta$
of length $n$ and $I\subseteq [n]$, define
\begin{equation}
\label{e:hJeta}
h_I(\eta)=\left|\{(r<s) :  r\in I, s\not\in I,\eta_s=\eta_{r}+1\}\right|\,,
\end{equation}
where $(r<s)$ denotes a pair of positions $(r,s)$ in $\eta$ with
$1\leq r<s\leq n$.

Our reformulation of \eqref{eq:EE} is stated in the following theorem, proven
at the end of this section.

\begin{thm}\label{t:rhs_reformulation}
For $0\leq l < m\leq N$, we have
\begin{multline}
\label{e:rhs_reformulation} \langle z^{N-m}\rangle
\sum_{\substack{\lambda\in\Dyck_{N}\\ P\in\LD_{N,l}(\lambda)}}
t^{|\delta/\lambda|}\,\prod_{\substack{1<i\leq N\\
c_i(\lambda)=c_{i-1}(\lambda)+1}}
(1+z\,t^{-c_i(\lambda)}) q^{{\rm dinv}(P)}x^{\wt_+(P)} \\
= \sum_{\substack{J \subseteq [m-1]\\ |J|=l}} \,
\sum_{\substack{\tau,(0,\aA)\in \mathbb N^{m}\\ |\tau|=N-m}} t^{|\aA|}
q^{d((0,\aA),\tau)+h_J(\aA)} 
N_{((0,\aA)+(1^{m})+\tau )/((\aA,0)+\varepsilon_J)}(X;q)\,,
\end{multline}
where 
$N_{\beta/\alpha}$ is given by
Definition~\ref{d:F}, below.
\end{thm}

\subsection{Definition of
\texorpdfstring{$N_{\beta/\alpha}$}{N\_\textbeta/\textalpha}}
\label{ss:N-beta-alpha}

For $\alpha ,\beta \in \ZZ ^{l}$ such that $\alpha _{j}\leq \beta
_{j}$ for all $j$,
define $\beta /\alpha $ to be the tuple of single
row skew shapes $(\beta _{j})/(\alpha _{j})$ such that the $x$
coordinates of the right edges of boxes $a$ in the $j$-th row are the
integers $\alpha _{j}+1,\ldots,\beta _{j}$.  
The boxes just outside
the $j$-th row, adjacent to the left and right ends of the row, then
have $x$ coordinates $\alpha _{j}$ and $\beta _{j}+1$.  
We consider these two boxes to be adjacent to the ends of an empty
row, with $\alpha _{j} = \beta _{j}$, as well.

Given a tuple of skew row shapes \(\beta/\alpha\), three boxes
\((u,v,w)\) form a \emph{\(w_0\)-triple} when box $v$ is in row $r$ of
$\beta/\alpha$, boxes \(u\) and \(w\) are in or adjacent to a row $j$
with $j>r$, and the $x$-coordinates $i_u, i_v, i_w$ of these boxes
satisfy $i_u=i_{v}$ and $i_w=i_v+1$.  These triples are a special case
of \(\sigma\)-triples defined for any \(\sigma \in S_l\)
in~\cite{paths}.  We denote the number of $w_0$-triples in
$\beta/\alpha$ by $h_{w_0}(\beta/\alpha)$.  The reader can verify that
\begin{equation}\label{def:triples}
h_{w_0}(\beta/\alpha) = \sum_{r < j} \big| [\alpha_r+1,\beta_r] \cap
[\alpha_j, \beta_j] \big| \,.
\end{equation}

For a totally ordered alphabet \(\Acal\), a {\it row strict tableau}
of shape $\beta/\alpha$ is a map \(S \colon \beta/\alpha \to \Acal\)
that is strictly increasing on each row.  The set of these maps is
denoted by $\NSYT(\beta/\alpha,\Acal)$. For convenience, 
given $\alpha ,\beta \in \ZZ
^{l}$ with some $\alpha _{j} > \beta_{j}$,  we set 
$\NSYT(\beta/\alpha,\Acal) = \varnothing$.

A \(w_0\)-triple $(u,v,w)$ is an \emph{increasing} \(w_0\)-triple in
\(S\) if \(S(u) < S(v) < S(w)\), with the convention that \(S(u) =
-\infty\) if \(u\) is adjacent to the left end of a row of
\(\beta/\alpha\), and \(S(w) = \infty\) if \(w\) is adjacent to the
right end of a row. Let \(h_{w_0}(S)\) be the number of increasing
\(w_0\)-triples in $S$.

For \(S \in \NSYT(\beta/\alpha,\NN)\), define 
\begin{equation}
x^{\wt_+(S)} = \prod_{u \in \beta/\alpha, \, S(u) \neq 0}
x_{S(u)}\qquad\text{and}\qquad x^{\wt(S)} = \prod_{u \in \beta/\alpha}
x_{S(u)}\,.
\end{equation} 

\begin{figure}
    \centering
    \begin{tikzpicture}[scale=.43]
      \newcounter{c};
      \newcommand{\drawRow}[4]{ \def\r{#1}; \def\a{#2}; \def\b{#3};
        \def\labels{#4}; \draw [shift={(0,1.5*\r)}] (\b+2,1) grid
        (\a+2,0); \node at (\a+1,1.5*\r+.5) {$-\infty$}; \node at
        (\b+2.5,1.5*\r+.5) {$\infty$}; \setcounter{c}{\a+1}; \foreach
        \nn in \labels {\node at (\thec+1.5,1.5*\r+.5) {
            $\footnotesize \nn$};
          \addtocounter{c}{1}; } };
      \node at (-1,8) {$S=$};
      \drawRow{10}{0}{3}{1,3,4} \drawRow{9}{2}{3}{0}
      \drawRow{8}{2}{2}{} \drawRow{7}{1}{3}{3,5} \drawRow{6}{2}{2}{}
      \drawRow{5}{1}{1}{} \drawRow{4}{0}{1}{2} \drawRow{3}{0}{1}{1}
      \drawRow{2}{0}{2}{0,6} \drawRow{1}{1}{2}{4} \drawRow{0}{1}{1}{}
      \node at (8.5,8) {$\,$};
    \end{tikzpicture}
\caption{\label{fig:neg-tableau}
For $\beta=(12211123233)$, $\alpha=(11000121220)$, there are
$h_{w_0}(\beta/\alpha)=29$ $w_0$-triples in $\beta/\alpha$.  The row
strict tableau $S$ of shape $\beta/\alpha$ has $h_{w_0}(S)=15$
increasing $w_0$-triples, \(x^{\wt_+(S)} \!= x_1^2 x_2 x_3^2 x_4^2 x_5
x_6\), and \(x^{\wt(S)} = x_0^2 x_1^2 x_2 x_3^2 x_4^2 x_5 x_6\).  }
\end{figure}

\begin{defn}\label{d:F}
For $\alpha ,\beta \in \NN^{m}$, define
\begin{equation}
\label{e:Lcalenx} N_{\beta/\alpha} =
N_{\beta/\alpha}(X;q) = \sum_{S \in \NSYT(\beta/\alpha,
\ZZ_+)} q^{h_{w_0}(S)} x^{\wt(S)}\,.
\end{equation}
Note that if $\alpha _{j}> \beta _{j}$ for any $j$ then
$N_{\beta/\alpha}=0$ by our convention that $\NSYT(\beta/\alpha,\Acal)
= \varnothing$.
\end{defn}

\begin{remark}\label{r:Nfinitesuffices}
It is shown in~\cite[Proposition 4.5.2]{paths} and its proof that, for
$\alpha ,\beta \in \NN^{m}$, $N_{\beta/\alpha}$ is a symmetric
function whose Schur expansion involves only $s_\lambda$ where
$\ell(\lambda)\leq m$.
\end{remark}

\subsection{Transforming the combinatorial side}
\label{ss:transform-combinatorial}

To prove~\eqref{e:rhs_reformulation}, we first associate
each Dyck path with a tuple of row shapes recording vertical runs.

\begin{defn}\label{def:LLT-data}
The {\em LLT data} associated to a path $\lambda\in\Dyck_N$ is
\begin{equation*}
\beta= (1,c_2(\lambda)+1,\ldots,c_N(\lambda)+1)
\;\;\text{and}\;\;\alpha= (c_2(\lambda),\ldots,c_N(\lambda),0)\,,
\end{equation*}
where $c_{i}(\lambda)$ counts lattice squares between $\lambda$ and
the line segment connecting $(0,N)$ to $(N,0)$ in column $i$, numbered
from right to left, as in Lemma~\ref{lem:r-prod=c-prod}.
\end{defn}

Figure~\ref{fig:neg-tableau} shows the LLT data $\beta,\alpha$
associated to the path \(\lambda\) in Figure~\ref{fig:PF}.  Note that
$\beta_i$ (resp. $\alpha_i$) is the furthest (resp. closest) distance
from the diagonal to the path $\lambda$ on the line $x = N-i$, so that
$\beta_i-\alpha_i$ is the number of south steps of $\lambda$ on that
line.

This association allows us to relate $q$-weighted sums over
partial labellings to the $N_{\beta/\alpha}$.

\begin{lemma}\label{cor:pld2negtab}
For \(\lambda \in \Dyck_{N}\) and its associated LLT data
\(\alpha,\beta\), we have
\begin{equation}
\sum_{P\in\LD_{N,l}(\lambda)} q^{\dinv(P)} x^{\wt_+(P)} =
\sum_{\substack{I\subseteq [N-1] \\ |I|=l}} q^{h_I(\alpha)}
N_{\beta/(\alpha+\varepsilon_I)}(X;q)\,.
\end{equation}
\end{lemma}

\begin{proof}
There is a natural weight-preserving bijection mapping \(P \in
\LD_N(\lambda)\) to \(S \in \NSYT(\beta/\alpha,\NN)\), where the
labels of column \(x=i\) of \(P\), read north to south, are placed
into row \(N-i\) of \(\beta/\alpha\), west to east.  See
Figures~\ref{fig:PF} and \ref{fig:neg-tableau}. Moreover, $\dinv(P) =
h_{w_0}(S)$.  To see this, let $\hat P$ be the same labelling as $P$
but with the ordering on letters taken to be $0>1>2\cdots$.  It is
proven in~\cite[Proposition 6.1.1]{paths} that \(\dinv_1(\hat P) =
h_{w_0}(S)\), where $\dinv_1(\hat P)$ was introduced in~\cite{HHLRU}
and matches $\dinv(P)$ as discussed in
Remark~\ref{r:reconcilenotation}.  The bijection restricts to a
bijection from $\LD_{N,l}(\lambda)$ to the subset of tableaux $S \in
\NSYT(\beta/\alpha,\NN)$ with exactly $l$ 0's, none in row $N$.  This
gives
\begin{equation}\label{e:nozeros}
\sum_{P\in\LD_{N,l}(\lambda)} q^{\dinv(P)} x^{\wt_+(P)} =
\sum_{\substack{I\subseteq  [N-1]\\ |I|=l }}
\sum_{\substack{S\in\NSYT(\beta/\alpha, \NN)\\ 0\text{ in rows }i\in
I}} q^{h_{w_0}(S)} x^{\wt_+(S)} \,.
\end{equation}
The claim then follows from Definition~\ref{d:F} and the following Lemma.
\end{proof}

\begin{lemma}\label{l:dropzeroontriples}
For $\alpha ,\beta \in \NN^{N}$ and \(S \in \NSYT(\beta/\alpha,\NN)\),
let $I\subseteq [N]$ be the rows of $S$ containing a zero and let $T$ be
the tableau in \(\NSYT(\beta/(\alpha+\varepsilon_I),\ZZ_+)\) obtained
by deleting all zeros from $S$.  Then
\begin{equation}\label{eq:dropping-zeros-on-triples}
h_{w_0}(T) = h_{w_0}(S)-h_I(\alpha)\,,
\end{equation}
where $h_I(\alpha)$ is defined in~\eqref{e:hJeta}.
\end{lemma}

\begin{proof}
Consider an increasing \(w_0\)-triple $(u,v,w)$ of $S$; the entries
satisfy $S(u)<S(v)<S(w)$, $v$ lies in some row $r$, and both $u$ and
$w$ lie in a row $j>r$.  When $r\not\in I$, either $j\not\in I$ so
that $(u,v,w)$ is an increasing \(w_0\)-triple of $T$ with the same
entries as $S$, or $j\in I$ and $S(u)=0$ changes to $T(u)=-\infty$
where still $(u,v,w)$ is an increasing \(w_0\)-triple of $T$.
However, if $r\in I$, $S(v)=0$ changes to $T(v)=-\infty$ and thus
$(u,v,w)$ is not an increasing \(w_0\)-triple of $T$.  Note the
increasing condition implies that this happens only when $j\not\in I$
and $\alpha_r=\alpha_j-1$ since $S(u)<0<S(w)$.
Thus~\eqref{eq:dropping-zeros-on-triples} follows.
\end{proof}

\begin{defn}
\label{d:betaalpha}
Given $\aA=(a_1,\ldots,a_{m-1})\in\NN^{m-1}$ and
$\tau = (\tau_1,\ldots,\tau_{m}) \in \NN^{m}$,
we define two sequences $\beta _{\aA \tau }$ and $\alpha
_{\aA \tau }$ of length $|\tau |+m$ as follows.

The sequence $\beta _{\aA \tau }$ is the concatenation of
sequences $(1,2,\ldots,\tau _{1}+1)$ and $(a_{i-1}+1, a_{i-1}+2,
\ldots, a_{i-1}+\tau _{i}+1)$ for $i=2,\ldots,m$.  The sequence
$\alpha _{\aA \tau }$ is the same as $\beta _{\aA \tau }$ except
in the positions corresponding to the ends of the concatenated
subsequences.  In these positions, we change the entries $\tau
_{1}+1, a_{1}+\tau _{2}+1 ,\ldots,a_{m-1}+\tau _{m}+1 $ in $\beta
_{\aA \tau }$ to $a_{1}, a_{2},\ldots, a_{m-1}, 0$.  
Equivalently, $\alpha _{\aA \tau }$ is the same as the sequence
obtained by subtracting $1$ from all entries of $\beta _{\aA \tau }$
and shifting one place to the left, deleting the first entry and
adding a zero at the end.
\end{defn}

\begin{example}\label{e:agamma}
For $\aA=(130012)$ and $\tau = (2311022)$,
\begin{equation}\label{e:alpha-beta-example}
\setlength{\arraycolsep}{2pt}
\begin{array}{rcccc@{\hskip 12pt}cccc@{\hskip 12pt}cc@{\hskip 12pt}cc@{\hskip 12pt}c@{\hskip 12pt}ccc@{\hskip 12pt}ccc}
(0,\aA)+(1^m)+\tau & = & (\phantom{1}&&{3}&&&&{5}&&{5}&&{2}&{1}&&&{4}&&&{5}) \\
\beta_{\aA\tau}
  & = &(1&2&3&2&3&4&5&4&5&1&2&1&2&3&4&3&4&5)\\
\alpha_{\aA\tau}
  &= &(1&2&1&2&3&4&3&4&0&1&0&1&2&3&2&3&4&0)\\
(\aA,0)
  & = &(\phantom{1}&&{1}&&&&{3}&&{0}&&{0}&{1}&&&{2}&&&{0})
\end{array}\,.
\end{equation}
The wider spaces show the division into blocks of size \(\tau_i+1\).
The last entry of \(\alpha_{\aA \tau}\) in each block is \(a_i\), and
the next block in $\alpha _{\aA \tau }$ and $\beta _{\aA \tau }$
starts with \(a_i+1\).
\end{example}

\begin{lemma}
For $0\leq l<m\le N$, 
\begin{multline}\label{l:alllabels}
\langle z^{N-m}\rangle\!\!\!  \sum_{\substack{\lambda\in \Dyck _{N}\\
P\in\LD_{N,l}(\lambda)}} t^{|\delta/\lambda|} \prod_{\substack{1 <i
\leq N \\ c_i(\lambda) = c_{i-1}(\lambda)+1}}
(1+z\,t^{-c_i(\lambda)})\, q^{\dinv(P)} x^{\wt_+(P)} \\
= \sum_{\substack{I\subseteq [N-1]\\ |I|=l}} \, \sum_{\substack{\tau, \,
(0,\aA)\, \in \, \NN^{m}\\ |\tau|=N-m}} t^{|\aA|}
q^{h_I(\alpha_{\aA\tau})}
N_{\beta_{\aA\tau}/(\alpha_{\aA\tau}+\varepsilon_I)}(X;q)\,.
\end{multline}
\end{lemma}

\begin{proof}
Use Lemma~\ref{cor:pld2negtab} to rewrite the left hand side of
\eqref{l:alllabels} as
\begin{equation}\label{e:labenL}
\langle z^{N-m}\rangle \sum_{\lambda\in\Dyck_{N}} t^{|\delta/\lambda|}
\prod_{\substack{1 <i \leq N \\ c_i(\lambda) = c_{i-1}(\lambda)+1}}
(1+z\,t^{-c_i(\lambda)}) \sum_{\substack{I\subseteq [N-1]\\ |I|=l}}
q^{h_I(\alpha)}N_{\beta/(\alpha+\varepsilon_I)}
\end{equation}
where $\beta= (1^{N})+(0,c_2(\lambda),\ldots,c_{N}(\lambda))$, $\alpha
= (c_2(\lambda),\ldots,c_{N}(\lambda),0)$ are the LLT data for
$\lambda$.  Note that a tuple $\mathbf c=(c_1,c_2,\ldots,c_{N}) \in
\NN^N$ is the sequence of column heights $c_{i}(\lambda )$ of
a path $\lambda\in\Dyck_{N}$ if and only if $c_s\leq
c_{s-1}+1$ for all $s>1$ and $c_1=0$; in this case,
$|\delta/\lambda|=|\mathbf c|$.  Replace $\Dyck_{N}$
in~\eqref{e:labenL} by these tuples, and expand the product to obtain
\begin{multline}\label{e:almost}
\langle z^{N-m}\rangle
\sum_{A\subseteq [N]\setminus \{1\}}\,\,
\sum_{\substack{c_i \leq c_{i-1}+1 \ \forall i  \\
c_i = c_{i-1}+1 \ \forall i \in A}} \!\!\!  t^{|\mathbf
c|-\sum_{i\in A}c_i} \, z^{|A|} \sum _{\substack{I\subseteq [N-1]\\
|I|=l}}
q^{h_I(\alpha)}\,N_{\beta/(\alpha+\varepsilon_I)}\\
= \sum_{\substack{\{1\} \subseteq J\subseteq [N] \\ |J| = m}} \,\,
\sum_{c_j = c_{j-1}+1 \ \forall j \notin J } \!\! t^{\sum_{j\in J}c_j}
\sum_{\substack{I\subseteq [N-1]\\ |I|=l}}
q^{h_I(\alpha)}\,N_{\beta/(\alpha+\varepsilon_I)}\,,
\end{multline}
where the equality comes from re-indexing with $J = [N]\setminus A$ and
noting that we can drop the condition $c_j \leq c_{j-1}+1 \ \forall
j\in J$ because $N_{\beta/(\alpha+\varepsilon_I)}=0$ if any
$(\alpha+\varepsilon_I)_j\geq \alpha_j>\beta_j$.

If we replace the sum over $J$ by a sum over $\{\tau\in\NN^m :
|\tau|=N-m\}$ using $J=
\{1,\tau_1+2,\tau_1+\tau_2+3,\ldots,\tau_1+\cdots+\tau_{m-1}+m\}$,
then, for fixed $J$ (or fixed $\tau$), the sum over $\mathbf{c}$ can
be replaced by a sum over
\begin{equation}\label{e:c-vector}
\mathbf c =
(0,1,2,\ldots,\tau_1,a_1,a_1+1,\ldots,a_1+\tau_2,a_2,\ldots,
a_{m-1}+\tau_{m})
\end{equation}
for $\mathbf a$ ranging over $\NN^{m-1}$.  Note that $\sum_{j\in J}
c_j = |\aA|$.  With this encoding of $\mathbf{c}$, we have
$\beta/\alpha = \beta_{\aA\tau}/\alpha_{\aA\tau}$ in the notation of
Definition \ref{d:betaalpha}, and \eqref{e:almost} becomes the right
hand side of \eqref{l:alllabels}.
\end{proof}

We make a final adjustment to the right hand side
of~\eqref{l:alllabels}.  This sum runs over tuples
$\beta_{\aA\tau}/(\alpha_{\aA\tau}+\varepsilon_I)$ with $|\tau|$
necessarily empty rows which can be removed at the cost of a $q$
factor. We introduce some notation depending on a given
$\aA\in\NN^{m-1}$, $\tau=(\tau_1,\ldots,\tau_{m})\in\NN^{m}$, and the
associated $\beta_{\aA\tau}/\alpha_{\aA\tau}$ from
Definition~\ref{d:betaalpha}.  First we set
$j_{\uparrow}=j+\sum_{x\leq j}\tau_x$ for $j\in [m]$, so the entry of
$\beta _{\aA \tau }$ in position $j_{\uparrow }$ is $a_{j-1}+\tau
_{j}+1$, or $\tau _{1}+1$ if $j = 1$, and the entry of $\alpha _{\aA
\tau }$ in the same position is $a_{j}$, or $0$ if $j = m$.  For a
subset $J\subseteq [m]$, we set $J_{\uparrow} = \{j_{\uparrow}: j\in
J\}$.  In positions $i\not\in[m]_{\uparrow}$, the sequences $\beta
_{\aA \tau }$ and $ \alpha _{\aA \tau }$ agree, so row $i$ is empty in
$\beta_{\aA\tau} /\alpha_{\aA\tau} $.  The tuple of row shapes
obtained by deleting these empty rows from $\beta_{\aA
\tau}/\alpha_{\aA \tau}$ is $((0,\aA)+(1^m)+\tau)/(\aA,0)$, where row
$j\in[m]$ corresponds to row $j_{\uparrow}$ of $\beta_{\aA
\tau}/\alpha_{\aA \tau}$; note that rows \((j-1)_\uparrow\) and
\(j_\uparrow\) are separated by \(\tau_j\) empty rows. See
Figure~\ref{f:betaalpha}.

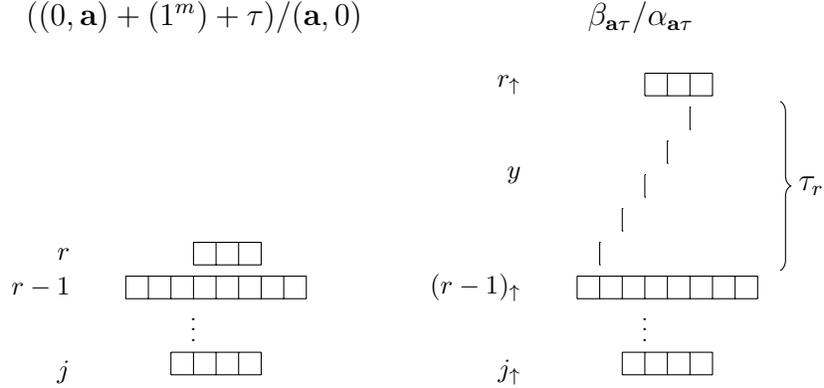
\begin{figure}
  \centering
  \begin{tikzpicture}[scale=.30]
    \newcommand{\drawRowpp}[3]{
    \def\r{#1};
    \def\a{#2};
    \def\b{#3};
    \draw [shift={(-20,1.5*\r)}] (\b+2,1) grid (\a+2,0);
  };
  \node[right] at (-22.9,17) 
{$((0,\aA)+(1^m)+\tau)/(\aA,0)$};
  \node[left] at (-20,6.5) {\footnotesize\(r\)};
  \drawRowpp{4}{3}{6}
  \node[left] at (-20,5) {\footnotesize\(r-1\)};
  \drawRowpp{3}{0}{8}
  \node at (-15,3.5) {\tiny\vdots};
  \node[left] at (-20,1.25) {\footnotesize\(j\)};
  \drawRowpp{.75}{2}{6}
    \newcounter{cp};
    \newcommand{\drawRowp}[4]{
    \def\r{#1};
    \def\a{#2};
    \def\b{#3};
    \def\labels{#4};
    \draw [shift={(0,1.5*\r)}] (\b+2,1) grid (\a+2,0);
    \setcounter{cp}{\a+1};
    \foreach \nn in \labels {\node at
      (\thecp+1.5,1.5*\r+.5) {
        \footnotesize{\nn}};
      \addtocounter{cp}{1};
    }
  };
  \node[right] at (2,17)
  {\(\beta_{\aA\tau} / \alpha_{\aA\tau}\)};
  \node[left] at (0,14) {\footnotesize\(r_\uparrow\)};
 \drawRowp{9}{3}{6}{}
  \drawRowp{8}{5}{5}{}
  \node[left] at (0,10) {\footnotesize\(y\)};
  \drawRowp{7}{4}{4}{}
  \drawRowp{6}{3}{3}{}
  \drawRowp{5}{2}{2}{}
  \drawRowp{4}{1}{1}{}
  \node[left] at (0,5) {\footnotesize\((r-1)_\uparrow\)};
  \drawRowp{3}{0}{8}{}
  \node at (5,3.5) {\tiny\vdots};
  \node[left] at (0,1.25) {\footnotesize\(j_\uparrow\)};
  \drawRowp{.75}{2}{6}{}

  \draw [decorate,decoration={brace,amplitude=4pt}] (11,13.25) --
(11,5.75) node [midway, right,xshift=0.1cm] {\(\tau_r\)};
\end{tikzpicture}
\caption{\label{f:betaalpha}
Comparing the tuples of rows $\beta_{\aA \tau}/\alpha_{\aA \tau}$ and
$((0,\aA)+(1^m)+\tau)/(\aA,0)$ for \(\aA \in \NN^{m-1}\) and \(\tau
\in \NN^m\).  Here $a_j=2$, \(a_{r-1} = 0, a_r = 3\), and
\(\tau_r = 5\).}
\end{figure}

\begin{lemma}
\label{l:GbigisGsmall}
For $J\subseteq [m]$, $\aA\in\NN^{m-1}$ and $\tau\in\NN^m$, 
let $I=J_{\uparrow}$. Then
\begin{equation}
\label{e:pain} N_{\beta_{\aA\tau}/(\alpha_{\aA\tau}+\varepsilon_I)} =
q^{d((0,\aA),\tau)-h'_J(\aA,\tau)}
N_{((0,\aA)+(1^m)+\tau)/((\aA,0)+\varepsilon_J)}\,,
\end{equation}
where $h'_J(\aA,\tau)=\left|\{(j<r): j\in J, r\in[m],a_j\in
[a_{r-1},a_{r-1}+\tau_r-1]\}\right|$ with $a_0=0$, and
$d((0,\aA),\tau)$ is defined by \eqref{e:detagamma}.
\end{lemma}

\begin{proof}
Set $a_0=0$.  We can assume \(a_j+(\varepsilon_J)_j \leq
a_{j-1}+\tau_j+1\) for all $j\in[m]$ since otherwise both sides
of~\eqref{e:pain} vanish by Definition~\ref{d:F}.  Hence, each side is
a \(q\)-generating function for row strict tableaux on tuples of
single row skew shapes; rows of
$\beta_{\aA\tau}/(\alpha_{\aA\tau}+\varepsilon_I)$ on the left hand
side differ from the right hand side only by the removal of empty rows
$r\not\in[m]_{\uparrow}$.  Thus, the two sides agree up to a factor
$q^{d}$, where $d$ counts $w_0$-triples of
$\beta_{\aA\tau}/(\alpha_{\aA\tau}+\varepsilon_I)$ involving one of
these empty rows.

To evaluate $d$, consider such an empty row $(b)/(b)$, coming from
\(b\in\{a_{r-1}+1,\ldots,a_{r-1}+\tau_r\}\) for some $r\in[m]$.  The
adjacent boxes on the left and right of this empty row form a
$w_{0}$-triple, increasing in every tableau, with one box in each
non-empty lower row $j_{\uparrow}$, of the form
$(a_{j-1}+\tau_j+1)/(a_j+(\varepsilon_J)_j)$, such that $b\in
[a_j+(\varepsilon_J)_j +1,a_{j-1}+\tau_j+1]$.  Hence,
\begin{multline*} 
d= \sum_{1\leq j<r\leq m} \big| [a_j+(\varepsilon_J)_j
,a_{j-1}+\tau_j] \cap [a_{r-1},a_{r-1}+\tau_r-1] \big|
\\
= \sum_{1\leq j<r\leq m} \big| [a_j,a_{j-1}+\tau_j] \cap
[a_{r-1},a_{r-1}+\tau_r-1] \big| - \sum_{\substack{1\leq j<r\leq m
\\ j\in J}} \big|\{a_j\}\cap [a_{r-1},a_{r-1}+\tau_r-1] \big|.
\end{multline*}
The sum after the minus sign is $h_J'(\aA,\tau)$.
To prove that the remaining sum is \(d((0,\aA),\tau)\),
first rewrite it as
\begin{equation}\label{e:intervals intersection}
\! \sum_{1\leq j<r \leq m} \!\!\! \left( \big| [a_j,\infty) \cap
[a_{r-1},a_{r-1}+\tau_r-1] \big|- \big| [a_{j-1}+\tau_j+1,\infty)
\cap [a_{r-1},a_{r-1}+\tau_r-1] \big|\right),\!
\end{equation}
using the fact that $a_{j}\leq a_{j-1}+\tau _{j}+1$ by
assumption.  Next observe that since $a_{0} = 0\leq a_{r-1}$,
\begin{equation*}
\big| [a_{r-1},\infty) \cap [a_{r-1},a_{r-1}+\tau_r-1] \big| =
\big| [a_0,\infty) \cap [a_{r-1},a_{r-1}+\tau_r-1] \big|.
\end{equation*}
Adding $\sum_{1<j<r} \big| [a_{j-1},\infty) \cap
[a_{r-1},a_{r-1}+\tau_r-1] \big|$ to both sides, it follows that
\begin{equation*}
\sum_{1\leq j<r} \big| [a_{j},\infty) \cap
[a_{r-1},a_{r-1}+\tau_r-1] \big| = \sum_{1\leq j<r} \big|
[a_{j-1},\infty) \cap [a_{r-1},a_{r-1}+\tau_r-1] \big|.
\end{equation*}
Hence \eqref{e:intervals intersection} is unchanged upon replacing
$[a_j,\infty)$ with $[a_{j-1},\infty)$ and is thus equal to
\begin{equation*}
\sum_{1\leq j<r\leq m} \big|[a_{j-1},a_{j-1}+\tau_j]\cap
[a_{r-1},a_{r-1}+\tau_r-1]\big| = d((0,\aA),\tau).
\end{equation*}
\end{proof}

\begin{proof}[Proof of Theorem~\ref{t:rhs_reformulation}] Consider a
summand $t^{|\aA|} q^{h_I(\alpha_{\aA\tau})}
N_{\beta_{\aA\tau}/(\alpha_{\aA\tau}+\varepsilon_I)}$ on the right
hand side of identity~\eqref{l:alllabels} for $I\subseteq [N-1]$,
$\aA\in\NN^{m-1}$, $\tau\in\NN^m$.  It vanishes unless $I=J_\uparrow$
for some $J\subseteq [m-1]$ since $N_{\beta/(\alpha+\varepsilon_I)}=0$
when $(\alpha+\varepsilon_I)_i>\beta_i$ for some index $i$. 

For
$I=J_{\uparrow}$, we can replace the summand with $t^{|\aA|}
q^{d((0,\aA),\tau) + h_{I}(\alpha_{\aA\tau}) - h'_J(\aA,\tau)}
N_{((0,\aA) + (1^{m})+\tau) / ((\aA,0)+\varepsilon_J)}$, by
Lemma~\ref{l:GbigisGsmall}.  It now suffices to prove that for
$\alpha=\alpha_{\aA\tau}$,
\begin{equation}
h_I(\alpha) = 
h'_J(\aA,\tau)+ h_J(\aA) \,.
\end{equation}
We recall that $N=m_\uparrow$ and note that $[N]\setminus I =([N]\setminus
[m]_{\uparrow} )\sqcup ([m]_\uparrow\setminus I) =([N]\setminus
[m]_{\uparrow}) \sqcup ([m]\setminus J)_{\uparrow}$.
Hence,
$h_I(\alpha)=\left| \{(x<y): x\in I, y\in[N]\setminus
I,\alpha_y=\alpha_x+1\} \right| =\left|S_1\right|+\left|S_2\right|$
for
\begin{align*}
S_1 & =  \{(x<y): x\in J_\uparrow,\, 
y \in [N]\setminus [m]_\uparrow,\, \alpha_y=\alpha_x+1\}\, ,\\
S_2 & = \{(x<y): x\in J_\uparrow,\, y\in ([m]\setminus J)_\uparrow,\,
\alpha_y=\alpha_x+1\}\,.
\end{align*}
Since $\alpha_{m_\uparrow}=0$ implies $(x<m_\uparrow)\not\in S_2$
for all $x<m_\uparrow$, we use that $a_u=\alpha_{u_\uparrow}$ for
every $u\in[m-1]$ to see that
\begin{equation}
h_J(\aA)= \big|S_2\big| = \big| \{(j<r): j\in J,\,  r\in [m-1]\setminus
J,\, a_r=a_j+1\} \big|\,.
\end{equation}
Furthermore, $\{(j<r):j\in J, r\in[m],a_{r-1}+1\leq a_j+1\leq
a_{r-1}+\tau_r\}$ and $S_1$ are equinumerous,
as we can see by letting a pair $(j<r)$ in the first set correspond to
the pair $(j_\uparrow<y)$ in $S_{1}$, where $y$ is the unique row
index in the range $(r-1)_\uparrow<y<r_\uparrow$
such that $\alpha_{y} = \alpha_{j_\uparrow} +1 = a_j+1$, as
illustrated in Figure~\ref{f:betaalpha}.
\end{proof}

\section{Stable unstraightened extended delta theorem}
\label{s:unstraightened}
\subsection{}

By Theorems~\ref{t:ourLHS} and~\ref{t:rhs_reformulation},
the {\bf Extended Delta Conjecture is equivalent to}
\begin{multline}\label{e:infseriesidentity}
\Hbold^{m}_q \left(
  \frac{\prod _{i+1<j\leq m} (1 - q\, t\, 
    x_i/x_j)}{\prod _{i< j\leq m} (1 - t\,
    x_i/x_j)} 
x_1\cdots x_{m} h_{N-m}(x_1,\ldots,x_{m}) \overline{e_l(x_2,\ldots,x_{m})}
\right)_{\rm pol}
\\
=
\sum_{\substack{J \subseteq [m-1]\\ |J|=l}} \ 
\sum_{\substack{(0,\aA),\tau\in \mathbb N^{m} \\ |\tau|=N-m}}
t^{|\aA|} q^{d((0,\aA),\tau)+h_J(\aA)}
\; \left(\omega N_{\beta/\alpha}\right)(x_1,\ldots,x_{m};q)
\,,
\end{multline}
where $\beta = (0,\aA)+(1^{m})+\tau$, $\alpha=(\aA,0)+\varepsilon_J$, 
and $\left(\omega N_{\beta/\alpha}\right)(x_1,\ldots,x_{m};q)$ is
$\omega N_{\beta/\alpha}(X;q)$ evaluated in $m$ variables.

Although this is an identity in only $m$ variables, it does amount to
the Extended Delta Conjecture by Remarks~\ref{r:sfsfinitesuffices}
and~\ref{r:Nfinitesuffices}: both $\omega ( h_l[B]e_{m-l-1}[B-1]
e_{N-l})$ and $\omega N_{\beta/\alpha}(X;q)$ for the $\alpha,\beta$
arising in~\eqref{e:infseriesidentity} are linear combinations of
Schur functions $s_\lambda$ with $\ell(\lambda)\leq m$.

We will show in Proposition~\ref{prop:llts2G} that the functions
$\omega N_{\beta/\alpha}$ on the right hand side
of~\eqref{e:infseriesidentity} are the polynomial parts of `LLT
series' introduced in~\cite{GrojHaim07}, making each side
of~\eqref{e:infseriesidentity} the polynomial part of an infinite
series of $\GL_m$ characters.  We then
prove~\eqref{e:infseriesidentity} as a consequence of a stronger
identity between these infinite series.

Hereafter, we use $x$ to abbreviate the alphabet $x_1,\ldots,x_m$.
\subsection{LLT series}
\label{S:LLTseries}

We will work with the (twisted) non-symmetric Hall-Littlewood
polynomials as in \cite{paths}.  For a $\GL _{m}$ weight $\lambda \in
\ZZ ^{m}$ and \(\sigma \in S_m\), the twisted non-symmetric
Hall-Littlewood polynomial $E^\sigma_\lambda(x;q)$ is an element of
$\ZZ[q^{\pm 1}][x_1^{\pm 1}, \dots, x_m^{\pm 1}]$ defined using an
action of the Hecke algebra on this ring; we refer the reader to
\cite[\S 4.3]{paths} for the precise definition, citing properties as
needed.  We also have their variants
\begin{equation}\label{eq:defEsigma}
F^\sigma_\lambda(x;q) = \overline{E^{\sigma w_0}_{-\lambda}(x;q)}\,,
\end{equation}
recalling that
$\overline{f(x_1,\ldots,x_m;q)}=f(x_1^{-1},\ldots,x_m^{-1};q^{-1})$.

For any weights $\alpha ,\beta \in \ZZ ^{m}$ and a permutation $\sigma
\in S_{m}$, the {\em LLT series} $\Lcal^{\sigma } _{\beta /\alpha
}(x;q) =\Lcal^{\sigma } _{\beta /\alpha }(x_{1},\ldots,x_{m};q)$ is
defined in \cite[\S4.4]{paths} by
\begin{equation}\label{e:seriesLLTdef}
\langle \chi_\lambda \rangle
\Lcal^{\sigma^{-1}}_{\beta/\alpha}(x;q^{-1}) = \langle E^\sigma_\beta
\rangle \,\chi_\lambda \cdot E^\sigma_\alpha\,.
\end{equation}
Alternatively,~\cite[Proposition 4.4.2]{paths} gives the following
expression in terms of the Hall-Littlewood symmetrization operator
in~\eqref{e:Hq}:
\begin{equation}
\label{e:L-formula}
\Lcal ^{\sigma }_{\beta /\alpha }(x;q) = \Hbold^{m}_q(w_{0}(F^{\sigma
^{-1} }_{\beta }(x;q) \overline{E^{\sigma ^{-1}}_{\alpha }(x;q)}))\,,
\end{equation}
where $w_0$ denotes the permutation of maximum length here and after.
We will only need the LLT series for $\sigma=w_0$ and $\sigma = id$,
although most of what follows can be generalized to any $\sigma$.

In addition to the above formulas, we have the following combinatorial
expressions for the polynomial truncations of LLT series as tableau
generating functions with $q$ weights that count triples.  As usual, a
\emph{semistandard tableau} on a tuple of skew row shapes
$\nu=\beta/\alpha$ is a map \(T \colon \nu \to [m]\) which is weakly
increasing on rows.  Let \(\SSYT(\nu)\) denote the set of these, and
define $x^{\wt(T)}=\prod_{b\in \nu}x_{T(b)}$.

\begin{prop}[{\cite[Remark 4.5.5 and Corollary 4.5.7]{paths}}]\label{prop:tableaux-form-for-llts}
If \(\alpha_i \leq \beta_i\) for all \(i\), then
\begin{equation}\label{eq:tableaux-for-Lcal-pol}
\Lcal^{w_0}_{\beta/\alpha}(x;q)_{\pol} = \sum_{T \in
\SSYT(\beta/\alpha)} q^{h'_{w_0}(T)} x^{\wt(T)} \,,
\end{equation}
where \(h'_{w_0}(T)\) is the number of \(w_0\)-triples \((u,v,w)\) of
\(\beta/\alpha\) such that \(T(u) \leq T(v) \leq T(w)\).
\end{prop}

\begin{prop}[{\cite[Proposition 4.5.2]{paths}}]
\label{prop:llts2G}
For any  $\alpha, \beta \in \ZZ^m$, 
\begin{equation}\label{e:llts2G}
\Lcal^{w_0}_{\beta/\alpha}(x;q)_{\pol} = \left(\omega
N_{\beta/\alpha}\right)(x;q) \,.
\end{equation}
\end{prop}

\subsection{Extended Delta Theorem}
\label{ss:extended-delta-theorem}

We now give several lemmas on non-symmetric Hall-Littlewood
polynomials, then conclude by using the Cauchy formula for these
polynomials to prove Theorem~\ref{thm:1}, below, yielding the
stronger series identity that implies~\eqref{e:infseriesidentity}.

\begin{lemma}\label{l:xtimesF}
For $\aA\in\mathbb N^{m-1}$ 
and $w_0\in S_m$ and $\tilde w_0\in S_{m-1}$ the 
permutations of  maximum length, we have
\begin{align}\label{Eminusvar}
E^{w_0}_{(\aA,0)}(x_1,\ldots,x_{m};q) & =
E_{\aA}^{\tilde w_0}(x_1,\ldots,x_{m-1};q)\\
\label{Eminusvar2}
F^{w_0}_{(0,\aA)}(x_1,\ldots,x_{m};q) & = F^{\tilde
w_0}_{\aA}(x_2,\ldots,x_{m};q)\,.
\end{align}
\end{lemma}

\begin{proof}
The factorizations $E_{(\aA,0)}^{w_0}(x_1,\ldots,x_m;q) =
E_{\aA}^{\tilde w_0}(x_1,\ldots,x_{m-1};q) E_{(0)}^{id}(x_m;q)$ and
$E_{(0,-\aA)}^{id}(x_1,\ldots,x_m;q) = E_{(0)}^{id}(x_1;q)
E_{-\aA}^{id}(x_2,\ldots,x_m;q)$ are given by ~\cite[Lemma
4.3.4]{paths}.  The claim then follows from the definition
$F_\aA^\sigma=\overline{E_{-\aA}^{w_0\sigma}}$ and noting that
\(E_{(0)}^{id}(x_{m};q)=1= F^{id}_{0}(x_1;q)\).
\end{proof}

Inverting all variables and specializing \(\sigma = w_0\) in \cite[Lemma 4.5.1]{paths} yields the following lemma.

\begin{lemma}
For $l\leq m$, \(\aA \in \ZZ^m\),
we have
\begin{equation}\label{Epieri}
\overline{e_l(x)}\,\overline{E_\aA^{w_0}(x;q)}
=\sum_{I\subseteq [m]:|I|=l} q^{h_{I}(\aA)}
\overline{E^{w_0}_{\aA + \varepsilon _{I}}(x;q)}\,,
\end{equation}
where
\(
h_I(\aA) =
\left| \{( i<j) \mid a_j=a_i+1,i\in I, j \notin I\} \right|\), as defined in~\eqref{e:hJeta}.
\end{lemma}

\begin{lemma}\label{l:F=E}
For every $\lambda \in \ZZ ^{m}$ and $\sigma \in S_{m}$, we have
\begin{equation}\label{e:F=E}
F^{\sigma }_{\lambda }(x;q) = w_{0} E^{w_{0}\sigma }_{w_{0} \lambda
}(x;q^{-1}).
\end{equation}
\end{lemma}

\begin{proof} The desired identity follows from
\begin{equation}\label{e:E=E}
w_{0}E^{\sigma }_{\lambda }(x_{1}^{-1},\ldots,x_{m}^{-1};q) =
E^{w_{0}\sigma w_{0}}_{-w_{0} \lambda }(x;q)
\end{equation}
by applying $w_{0}$ to both sides, substituting $\sigma \mapsto \sigma
w_{0}$, $\lambda \mapsto -\lambda $, and $q\mapsto q^{-1}$, and using
the definition of $F^{\sigma }_{\lambda }$.

To prove \eqref{e:E=E}, we use the characterization of $E^{\sigma
}_{\lambda }(x;q)$ by the recurrence \cite[(76)]{paths} and initial
condition $E^{\sigma }_{\lambda } = x^{\lambda }$ for $\lambda $
dominant.  The change of variables $x^{\mu }\mapsto x^{-w_{0}(\mu )}$
replaces the Hecke algebra operator $T_{i} = T_{s_{i}}$ in the
recurrence with $T_{w_{0}s_{i}w_{0}}$,
giving a modified recurrence
satisfied by the left hand side of \eqref{e:E=E}.  It is
straightforward to verify that the right hand side of \eqref{e:E=E}
satisfies the same modified recurrence.  Since both sides reduce to
$x^{-w_{0}(\lambda )}$ for $\lambda $ dominant, \eqref{e:E=E} holds.
\end{proof}

\begin{lemma}\label{l:E-vs-F-coefs}
Given $\alpha ,\beta \in \ZZ ^{m}$ and a symmetric Laurent polynomial
$f(x_{1},\ldots,x_{m})$, we have, for any $\sigma \in S_{m}$,
\begin{equation}\label{e:E-vs-F-coefs}
\langle E^{w_{0}\sigma w_{0}}_{w_{0}\beta }(x; q^{-1}) \rangle\,
f(x)\cdot E^{w_{0}\sigma w_{0}}_{w_{0}\alpha }(x; q^{-1}) = \langle
F^{\sigma }_{-\alpha }(x; q) \rangle\, f(x)\cdot F^{\sigma }_{-\beta
}(x; q).
\end{equation}
\end{lemma}

\begin{proof}
In fact, we will show that
\begin{equation}\label{e:E-vs-F-any}
\langle E^{w_{0}\sigma w_{0}}_{w_{0}\beta }(x; q^{-1}) \rangle\,
f(x)\cdot E^{w_{0}\sigma w_{0}}_{w_{0}\alpha }(x; q^{-1}) = \langle
F^{\sigma }_{-\alpha }(x; q) \rangle\, w_{0}(f(x)) \cdot F^{\sigma
}_{-\beta }(x; q),
\end{equation}
even if we do not assume that $f(x)$ is symmetric.  By
Lemma~\ref{l:F=E}, the right hand side of \eqref{e:E-vs-F-any} is
equal to
\begin{equation}\label{e:E-vs-F-RHS}
\langle E^{w_{0}\sigma }_{-w_{0}\alpha }(x;q^{-1}) \rangle \,
f(x)\cdot E^{w_{0}\sigma }_{-w_{0}\beta }(x;q^{-1}).
\end{equation}
By \cite[Proposition 4.3.2]{paths}, the functions $E^{\sigma
}_{\lambda }(x;q)$ and $E^{\sigma w_{0}}_{-\lambda }(x;q)$ are dual
bases with respect to to the inner product $\langle - , - \rangle_{q}$
defined there.  Moreover, it is immediate from the construction of the
inner product that multiplication by any $f(x)$ is self-adjoint.
This gives
\begin{equation}\label{e:E-vs-F-RHS-2}
\langle f(x) E^{w_{0}\sigma w_{0}}_{w_{0}\alpha }(x;q^{-1}),\,
E^{w_{0}\sigma }_{-w_{0}\beta }(x;q^{-1}) \rangle_{q^{-1}} = \langle
E^{w_{0}\sigma w_{0}}_{w_{0}\alpha }(x;q^{-1}),\, f(x) E^{w_{0}\sigma
}_{-w_{0}\beta }(x;q^{-1}) \rangle_{q^{-1}},
\end{equation}
in which the left hand side is equal to the left hand side of
\eqref{e:E-vs-F-any}, and the right hand side is equal to
\eqref{e:E-vs-F-RHS}.
\end{proof}

\begin{lemma}\label{l:hFPieri}
For $w_0$ the maximum length permutation in $S_m$ and $\eta\in\mathbb N^m$,
we have
\begin{equation}
h_l(x) F_\eta^{w_0}(x;q) = \sum_{\substack{\tau\in\mathbb N^m\\
|\tau|=l}} q^{d(\eta,\tau)} F_{\eta+\tau}^{w_0}(x;q)\,,
\end{equation}
recalling from~\eqref{e:detagamma} that $d(\eta,\tau) = \sum_{j < r}
\big| [\eta_{j},\eta_{j}+\tau_j] \cap [\eta_{r},\eta_{r}+\tau_r-1]
\big|$.
\end{lemma}

\begin{proof}
Set $\alpha = -\eta -\tau$ and $\beta = -\eta$.
By \eqref{e:seriesLLTdef} and Lemma~\ref{l:E-vs-F-coefs} (with $\sigma
=w_{0}$), we have
\begin{equation}\label{e:L-by-F}
\langle h_{l}(x) \rangle\, \Lcal ^{w_{0}}_{w_{0}(\beta /\alpha )}(x;q)
= \langle E_{w_0\beta }^{w_0}(x;q^{-1}) \rangle h_l(x) E_{w_0
\alpha}^{w_0}(x;q^{-1}) = \langle F_{- \alpha}^{w_0}(x;q) \rangle
h_l(x) {F_{- \beta }^{w_0}(x;q)}.
\end{equation}
By specializing all but one variable in
\eqref{eq:tableaux-for-Lcal-pol} to zero, Proposition
\ref{prop:tableaux-form-for-llts} implies that the coefficient of
$h_l$ in $\Lcal^{w_0}_{w_0(\beta /\alpha)}(x;q)_{\pol}$ is
\(q^{h'_{w_0}(T)}\) for \(T\) the semistandard tableau of shape
\(w_0(\beta/\alpha)\) 
filled with a single letter, where $h'_{w_0}(T)$
is the number of $w_0$-triples of $w_0(\beta /\alpha )=w_0(-\eta/(-\eta-\tau))$.  By \eqref{def:triples}, this number is $d(\eta,\tau)$.
\end{proof}

\begin{thm}\label{thm:1}
For $0\leq l<m\leq N$ and $w_0\in S_{m}$ the maximum length permutation, we
have
\begin{multline*}
\frac{\prod_{i+1<j\le m}(1-qtx_i/x_j)} {\prod_{i < j\le m} (1-tx_i/x_j)} x_1
\cdots x_{m} h_{N-m}(x_1,\ldots,x_{m})
\overline{e_l(x_2,\ldots,x_{m})} =\\
\sum_{\substack{(0,\aA),\tau\,\in\mathbb N^{m}\\ |\tau|=N-m}}
\sum_{\substack{I\subseteq [m-1]\\ |I|=l} }\, t^{|\aA|}
q^{d((0,\aA),\tau)+h_I(\aA)}\; w_0 \bigl(
F^{w_0}_{(0,\aA)+\tau+(1^{m})}(x_1,\ldots,x_{m};q)
\overline{E^{w_0}_{(\aA,0)+\varepsilon_I}(x_1,\ldots,x_{m};q)}\bigr).
\end{multline*}
\end{thm}
\begin{proof}
Our starting point is the Cauchy formula~\cite[Theorem 5.1.1]{paths}
for the twisted non-symmetric Hall-Littlewood polynomials associated
to any $\tilde \sigma\in S_{m-1}$:
\begin{equation}\label{e:Cauchy}
\frac{\prod _{i<j< m} (1 - q\, t\, x_{i} \, y_{j})}{\prod _{i\leq j <
m} (1 - t\, x_{i}\, y_{j})} = \sum _{\aA\in\mathbb N^{m-1}} t^{|\aA
|}\, E^{\tilde \sigma }_{\aA }(x_{1},\ldots,x_{m-1};q^{-1}) \,
F^{\tilde\sigma }_{\aA }(y_{1},\ldots,y_{m-1};q)\,.
\end{equation}
Take $\tilde \sigma = \tilde w_0$ the maximum length permutation in $S_{m-1}$,
replace $x_i$ by $x_i^{-1}$, and then let $y_j=x_{j+1}$ to get
\begin{equation}
\frac{\prod _{i+1<j\leq m}(1 - q\, t\, x_{j}/x_{i})}{\prod _{i<j\leq
m}(1 - t\, x_{j}/x_{i})} = \sum_{\aA\in\NN^{m-1}} t^{|\aA |} F^{\tilde
w_0}_{ \aA}(x_2,\ldots,x_{m};q) \overline{E^{\tilde
w_0}_{\aA}(x_1,\ldots,x_{m-1};q)}\,.
\end{equation}
Since $(x_1\cdots x_{m})F^{\tilde
w_0}_{\aA}(x_2,\ldots,x_{m};q)=(x_1\cdots
x_{m})F^{w_0}_{(0,\aA)}(x_1,\ldots,x_{m};q)= F^{w_0}_{
(0,\aA)+(1^{m})}(x_1,\ldots,x_{m};q)$ for $w_0\in S_{m}$
by~\eqref{Eminusvar2} and the definition of $F^\sigma$, we have

\begin{equation*}
\frac{\prod _{i+1<j\leq m}(1 - q\, t\, x_{j}/x_{i})}{\prod _{i<j\leq
m}(1 - t\, x_{j}/x_{i})} (x_1\cdots x_{m}) = \sum_{\aA\in\NN^{m-1}}
t^{|\aA |} F^{w_0}_{ (0,\aA)+(1^{m})}(x_1,\ldots,x_{m};q)
\overline{E^{\tilde w_0}_{\aA}(x_1,\ldots,x_{m-1};q)}\,.
\end{equation*}
Multiplying by $h_{N-m}(x_1,\ldots,x_{m})$ with the help of
Lemma~\ref{l:hFPieri} yields
\begin{multline*}
\frac{\prod _{i+1<j\le m}(1 - q\, t\, x_{j}/x_{i})}{\prod _{i<j\le
m}(1 - t\, x_{j}/x_{i})} (x_1\cdots x_{m})
h_{N-m}(x_1,\ldots,x_{m})\\
= \sum_{\substack{(0,\aA), \tau\in\NN^{m}\\|\tau|=N-m}} t^{|\aA |}
q^{d((0,\aA),\tau)} F^{w_0}_{\eta +\tau}(x_1,\ldots,x_{m};q)
\overline{E^{\tilde w_0}_{\aA}(x_1,\ldots,x_{m-1};q)}\,,
\end{multline*}
where $\eta=(1^{m})+(0,\aA)$ and we have used that
$d(\eta,\tau)=d((0,\aA),\tau)$ by~\eqref{e:detagamma}.  Now multiply by
$\overline{ e_l(x_1,\ldots,x_{m-1})}$ and apply~\eqref{Epieri} to get
\begin{multline}
\frac{\prod _{i+1<j\le m}(1 - q\, t\, x_{j}/x_i)}{\prod _{i<j\le m}(1
- t\, x_{j}/x_{i})} (x_1\cdots x_{m}) \overline{
e_l(x_1,\ldots,x_{m-1})}
h_{N-m}(x_1,\ldots,x_{m})\\
= \sum_{\substack{(0,\aA),\tau\,\in\mathbb N^{m}\\ |\tau|=N-m}} \,
\sum_{|I|=l} t^{|\aA |} q^{d((0,\aA),\tau)+h_I(\aA)} F^{w_0}_{\eta
+\tau}(x_1,\ldots,x_{m};q) \overline{E^{\tilde
w_0}_{\aA+\varepsilon_I}(x_1,\ldots,x_{m-1};q)}\,,
\end{multline}
where $I\subseteq [m-1]$.  The result then follows by
using~\eqref{Eminusvar} on the right hand side and applying $w_0\in
S_{m}$ to both sides, noting that $w_0(\overline{
e_l(x_2,\ldots,x_{m})})= \overline{ e_l(x_1,\ldots,x_{m-1})}$.
\end{proof}

\begin{proof}[Proof of the Extended Delta Conjecture]
It suffices to prove the reformulation in~\eqref{e:infseriesidentity};
this follows by applying $\Hbold^{m}_q$ and~\eqref{e:L-formula} to the
identity of Theorem~\ref{thm:1}, taking the polynomial part, and using
Proposition~\ref{prop:llts2G}.
\end{proof}

\bibliographystyle{hamsplain}
\bibliography{shuffle}

\end{document}